\documentclass[12pt]{amsart}
\usepackage[headings]{fullpage}

\usepackage{amsmath,amssymb,amsthm}

\usepackage{graphicx}
\usepackage{enumerate}
\usepackage{xcolor}
\usepackage{url}
\usepackage{slashed}
\usepackage{bm}
\usepackage{tikz-cd}
\usepackage{pb-diagram}
\usepackage[german,english]{babel}
 \usepackage{tikz}
 \usetikzlibrary{calc,angles,quotes}
\usepackage{enumitem}
\usepackage[bookmarks=true,%
colorlinks=true,%
linkcolor=blue,%
citecolor=blue,%
filecolor=blue,%
menucolor=blue,%
urlcolor=blue,%
breaklinks=true]{hyperref}

\usepackage[all]{xy}
\usepackage{verbatim}

\usepackage{algorithm}  
\usepackage{algpseudocode}
\usepackage{amsmath}     
\usepackage{amsfonts}    
\usepackage{multirow}
\usepackage{array}

\newtheorem{thm}{Theorem}[]
\newtheorem*{thm*}{Theorem}
\newtheorem{lem}[thm]{Lemma}
\newtheorem{prop}[thm]{Proposition}
\newtheorem{cor}[thm]{Corollary}

\newtheorem{rem}[thm]{Remark}
\newtheorem{defn}[thm]{Definition}

\newcommand{\param}{{\mathchoice{\mkern1mu\mbox{\raise2.2pt\hbox{$
					\centerdot$}}
			\mkern1mu}{\mkern1mu\mbox{\raise2.2pt\hbox{$\centerdot$}}\mkern1mu}{
			\mkern1.5mu\centerdot\mkern1.5mu}{\mkern1.5mu\centerdot\mkern1.5mu}}}

\renewcommand{\setminus}{{\smallsetminus}}

\author{Ni An, Jiexing Chen, Ingrid Irmer, Paul Schmutz Schaller}
\title{The Thurston spine and equivariant cell decompositions in genus 2}

\begin{document}
	\maketitle
	\begin{abstract}
		\noindent
		The Thurston spine of the genus 2 Teichm\"uller space of closed compact surfaces is obtained. In particular, we identify all mapping class group-orbits of filling sets of systoles in genus 2. Thereby, surprisingly rich cases already appear in this simplest case which is genus 2. This is contrasted with a cell decomposition of a bordification of Teichm\"uller space obtained computationally using Rivin's angle coordinates, on which the cells are labelled by Delaunay graphs closely related to our systolic graphs.
	\end{abstract}
	
	
	\section{Introduction}
	Every point $x$ of Teichm\"uller space $\mathcal{T}_{g}$ corresponds to a marked hyperbolic surface $S(x)$. On $S(x)$ there is at least one shortest geodesic; such geodesics are called \textit{systoles}. The \textit{Thurston spine} $\mathcal{P}_{g}$, \cite{Thurston} is the set of points in $\mathcal{T}_{g}$ at which the systoles cut the surface into polygons, i.e. the systoles fill the surface. It follows from \cite{Lojasiewicz1964} that $\mathcal{P}_{g}$ can be triangulated. As this paper is concerned with the genus 2 case, from now on $\mathcal{T}$ and $\mathcal{P}$ will be used to denote $\mathcal{T}_{2}$ and $\mathcal{P}_{2}$ respectively. A triangulation of $\mathcal{P}$ with interiors of simplices labelled by the sets of systoles is obtained. This is achieved mainly via theoretical arguments, with some numerical computations in some special cases. In the case of surfaces with punctures, spines of moduli space are obtained from cell decompositions of moduli space \cite{Harer} and \cite{PennerComplex}. In this paper we also survey what is known of related cell decompositions of moduli space of closed compact surfaces of genus 2, and give the details of an explicit computation based on Rivin's angle coordinates. 
	
	The function $syst:\mathcal{T}\rightarrow \mathbb{R}_{+}$ that takes every point $x$ to the length of the systoles on $S(x)$ is called the systole function. It is invariant under the action of the mapping class group on $\mathcal{T}$. Moreover, $syst$ is a topological Morse function, \cite{Akrout}, \cite{SchmutzMorse}. Topological Morse functions were first defined in \cite{Morse}. Informally these are continuous functions with isolated critical points, for which the level sets around regular or critical points of index $j$ are homeomorphic to the level sets around regular points or critical points of index $j$ respectively of a (smooth) Morse function. As shown in \cite{Morse}, topological Morse functions -- once they have been shown to exist --can be used in the same way as the usual smooth Morse functions for studying the topology of a manifold. In \cite{ConvexTMF} it is explained how to construct topological Morse functions. Our calculations show that $\mathcal{P}$ can be described as the union of unstable manifolds of the critical points of $syst$. It was shown in \cite{MS} that the Thurston spine always contains such a union, and examples from \cite{Ni} show that the Thurston spine sometimes has larger dimension than a union of unstable manifolds. For topological Morse functions, stable and unstable manifolds of critical point are not always uniquely defined, so we prefer to use the canonical object $\mathcal{P}$ in place of unstable manifolds of critical points.
	
	We show that there are exactly four different mapping class group-orbits of sets of geodesics that fill and do not contain any proper subsets that fill; these are called minimal filling sets. Minimal filling sets are interesting in that generically they label the highest dimensional strata of $\mathcal{P}$. In Subsection \ref{subdefns} we define four distinct sets of filling systoles, each with exactly four elements. These sets are called 4-star, 4-ring, 4-chain, and long $Y$.
\begin{prop}
\label{prop-minimal}
Every minimal filling subset of systoles is either a 4-star, a 4-ring, a 4-chain, or a long $Y$.
\end{prop}

For each of the 4-minimal filling sets it is known that there is an, up to isometry, unique surface of genus 2 for which the lengths of the four geodesics is minimal subject to the constraint that the lengths are all equal. However, for all of these 4 "minimal" surfaces, this set of four filling geodesics is not the set of systoles, see Theorem \ref{thm-irm}. Nevertheless, for 4-chains  there exist surfaces where this set is the set of systoles, although none of these points are critical points of the topological Morse function $syst$.
	
	\setlength{\unitlength}{2cm}
	\begin{figure}
		\label{systolicgraphs}
		\begin{picture}(6,3)(-1.5,-1.0)

			\put(-0.6,-0){\line(1,0){0.8}}
			\put(-0.2,-0.4){\line(0,1){0.8}}
			\put(-0.25,-0.05){$\bullet$}
			\put(-0.65,-0.05){$\bullet$}
			\put(0.15,-0.05){$\bullet$}
			\put(-0.25,0.35){$\bullet$}
			\put(-0.25,-0.45){$\bullet$}
			
			\put(0.6,0){\line(1,0){1.6}}
			\put(0.55,-0.05){$\bullet$}
			\put(0.95,-0.05){$\bullet$}
			\put(1.35,-0.05){$\bullet$}
			\put(1.75,-0.05){$\bullet$}
			\put(2.15,-0.05){$\bullet$}
			
			\put(2.6,-0.2){\line(1,0){0.4}}
			\put(2.6,0.2){\line(1,0){0.4}}
			\put(2.6,-0.2){\line(0,1){0.4}}
			\put(3.0,-0.2){\line(0,1){0.4}}
			\put(2.55,-0.25){$\bullet$}
			\put(2.95,-0.25){$\bullet$}
			\put(2.55,0.15){$\bullet$}
			\put(2.95,0.15){$\bullet$}
			
			\put(4,-0.6){\line(0,1){0.8}}
			\put(4,0.2){\line(1,1){0.3}}
			\put(4,0.2){\line(-1,1){0.3}}
			\put(3.95,-0.65){$\bullet$}  
			\put(3.95,-0.25){$\bullet$} 
			\put(3.95,0.15){$\bullet$} 
			\put(3.65,0.45){$\bullet$} 
			\put(4.25,0.45){$\bullet$}
			
		\end{picture}
		\caption{The systolic graphs (defined in Subsection \ref{subdefns}) of a  4-star, a 4-chain, a 4-ring, and a long $Y$. These are planar graphs in the quotient of the genus two surface by the hyperelliptic involution. Vertices of the graphs are the fixed points of the hyperelliptic involution and edges correspond to systoles.}
	\end{figure}
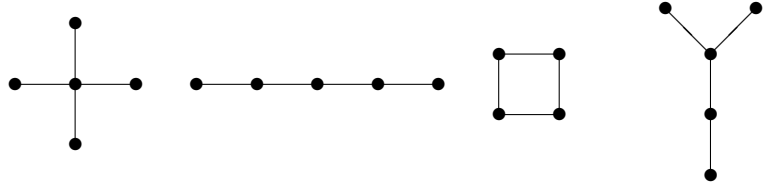
	
	The critical points of $syst$ and their indexes and automorphism groups are determined in \cite{SchmutzMorse}. These are depicted in Figure \ref{figcriticalpoints} together with their angle coordinates. Some of the unstable manifolds of these critical points are fixed point sets of finite subgroups of the mapping class group. Such fixed point sets are determined theoretically in Section \ref{sec:theory}. Thurston was the first to show that the critical points are contained in $\mathcal{P}$. This construction is explained in detail in \cite{PapT}. The fact that the critical points are contained in $\mathcal{P}$ also follows directly from convexity arguments that appeared after \cite{Thurston}.
	
	The theoretical arguments of Section \ref{sec:theory} completely determine the orbits of sets of systoles that are realised in the interiors of the top dimensional cells of $\mathcal{P}$. In addition to these theoretical arguments, numerical calculations are performed to determine the systoles in the lower dimensional cells. This relies on the fact that the sets of systoles are locally finite, i.e., for every $x\in \mathcal{T}$, there exists a neighbourhood $\mathcal{N}$ of $x$ in $\mathcal{T}$ such that the systoles are contained in the set of systoles of $S(x)$ at every point of $\mathcal{N}$. Local finiteness follows from the observation that for any $l\in \mathbb{R}_{+}$, there exist only finitely many geodesics on $S(x)$ of length less than or equal to $l$. This observation follows from \cite{M}; a simpler proof uses the collar lemma. 

Theorem \ref{newtheorem} together with the computations from Section \ref{spinecomputation} give the following:
\begin{thm}
\label{mainthm}
The quotient of the Thurston spine in genus 2 by the action of the mapping class group is a 3-dimensional cell complex with 3 top-dimensional cells. This is depicted in Figure \ref{All_cell}. The sets of systoles in the interiors of the 3-cells are all 4-chains.
\end{thm}
	
	
	There is a little known, surprisingly simple, construction of an equivariant cell decomposition of a bordification of $\mathcal{T}$ in \cite{SchmutzVoronoi} using sets of minima. This will be contrasted with a somewhat more complicated cell decomposition of another bordification of $\mathcal{T}$ obtained computationally. The latter cell decomposition is described theoretically in \cite{RivinII}, and is based upon a complete characterisation of the Delaunay decompositions of singular Euclidean surfaces, generalising the results of \cite{RivinI}. Related results were independently for example in \cite{Bowditch}, \cite{Garret}, \cite{Leibon} and \cite{Veech}. Our conclusion is that, at least in low genus, combinatorial descriptions of moduli space using the systole function are considerably simpler than those based on angle coordinates.
	
	The genus two case is special because each genus two surface has a hyperelliptic involution $\tau$ with six fixed points; these are called the Weierstrass points of the surface. This involution is also an isometry for the hyperbolic metric corresponding to a point in Teichm\"uller space. The study of the moduli space of genus two surfaces is thereby transformed into the problem of studying the moduli space of a sphere with six cone points. As pointed out in \cite{Armando}, Rivin's constraints can be used to determine the cell decompositions of the 2-sphere that can be realised as Voronoi decompositions centered on the Weierstrass points of $S(x)/\tau$ for some hyperbolic surface $S(x)$ corresponding to a point in moduli space. The 10 top-dimensional cells of such a cell decomposition, as well as some of the gluing maps, were determined in \cite{Armando}. The entire cell decomposition is computed here, as well as a cell decomposition of an induced bordification of $\mathcal{T}$, and a description of the four critical points of $syst$ in angle coordinates. 
	
	The cell decomposition using Voronoi decompositions centered on the Weierstrass points is closely related to the stratification of $\mathcal{T}$ by sets of systoles. The systolic graphs defined for sets of filling systoles, see for example Figure \ref{systolicgraphs}, are subgraphs of the Delaunay graphs of this cell decomposition. In genus greater than 2 we do not have the Weierstrass points to define a Voronoi/Delaunay decomposition, but the topological techniques using the stratification by sets of systoles, and the critical points of the systole function generalise. 

	This cell decomposition, as well as the spine, were shown to have the orbifold Euler characteristic $\frac{-1}{240}$ obtained in \cite{HarerZagier}.
	
	\textbf{Outline of the paper.} Section \ref{subdefns} is provided to introduce the assumptions, notations and some of the background to this subject used throughout the paper. Section \ref{sec:theory} provides theoretical arguments to completely determine the sets of filling systoles in the top-dimensional strata of $\mathcal{P}$. The algorithm and results of our calculation of the Thurston spine is given in Section \ref{spinecomputation}, and the final section surveys the two known equivariant cell decompositions of $\mathcal{T}$ as well as a description of the algorithm for obtaining the more complicated cell decomposition based on Rivin's angle coordinates. The Appendix gives the pseudocode for the computation of the computation of the cell decomposition based on angle coordinates.

\subsection*{Acknowlegments}
	The third author would like to thank Boris Springborn for help and advice in understanding the circle packing theory used in Subsection \ref{anglecoordssec}. 
	
	\subsection{Definitions and Background}
	\label{subdefns}
	
	The purpose of this subsection is to gather together the notation, definitions, and simple observations used throughout the paper.
	
	We denote by $S(x)$ a marked hyperbolic surface of genus 2 corresponding to a point $x$ in the Teichm\"uller space $\mathcal{T}$ of genus 2 surfaces. Subsurfaces of $S(x)$ are all assumed to be embedded with geodesic boundary. A $(g,n)$-\textit{subsurface} of $S$ has genus $g$, $g\in\{0,1\}$, and $n$ boundary geodesics, $n\in\{1,2,3,4\}$.
	
	A \textit{geodesic} is always assumed to be a simple closed geodesic in $S$. A \textit{systole} in $S$ is a shortest geodesic of $S$. The set of systoles of $S$ is denoted by $\mathrm{sys}(S)$. We say that a set $C$ of geodesics \textit{fills} if the complement of $C$ consists of topological disks. Following Thurston, we call the set of all points at which the set of systoles is exactly $C$ the \textit{stratum} labelled by $C$. The stratum labelled by $C$ will also be denoted by $\mathrm{str}(C)$.
	
	Let $S$ be such that $\mathrm{sys}(S)$ fills. Cut and paste arguments can be used to show that the elements of $\mathrm{sys}(S)$ are all non-separating with pairwise geometric intersection number at most one. In what follows, geodesics are assumed to be non-separating and with pairwise geometric intersection number at most one.
	
	Surfaces of genus 2 have a hyperelliptic involution with six fixed points. These six fixed points are called Weierstrass points. Every (simple, non-separating) geodesic passes through exactly two of them, \cite{Hyperelliptic}. For the sets of systoles we will be considering it will be shown that two geodesics cannot pass through the same pair of fixed points. Moreover, if two geodesics in the set intersect, they intersect in one of these six fixed points because otherwise they intersect at least twice.
	
	As a consequence, we may represent the set of systoles $\mathrm{sys}(S)$ by a graph with at most six vertices called the \textit{systolic graph} of $S$. Examples of systolic graphs are given in Figure \ref{systolicgraphs}. The vertices represent the fixed points of the hyperelliptic involution while the edges represent the systoles. Moreover the graph is connected if $\mathrm{sys}(S)$ fills. On the other hand, it will be shown that if the graph is connected and has at least five vertices, then $\mathrm{sys}(S)$ fills. We shall see in Lemma \ref{lem-4ring} that there is a particular case in which the graph is connected and has four vertices. This graph represents a set of systoles that might or might not fill.

For a set of nonseparating geodesics intersecting pairwise at most once, we will also speak of the systolic graph, constructed in the same way as for systoles. 
	
	The integers are used to label the fixed points of the hyperelliptic involution of $S$. A (non-separating) geodesic $a$ in $S$ joining $p$ and $q$ for different integers $p,q\in\{1,2,3,4,5,6\}$ will be denoted $p-q$.  We also write $a=p-q$ making no distinction between $p-q$ and $q-p$. The hyperelliptic involution is denoted by $\tau$.
	
	\begin{prop} \label{prop-graph} An automorphism $\psi$ of the surface $S$, i.e. an element of the mapping class group that fixes $S$ in Teichm\"uller space, induces an automorphism on the systolic graph of $S$.
	\end{prop}
	\begin{proof} The automorphism $\psi$ induces an automorphism of $\mathrm{sys}(S)$ and the proposition follows.
	\end{proof}
	
Note that the converse to Proposition \ref{prop-graph} does not hold, because a graph automorphism might not preserve the embedding of the graph in the surface, in which case it would not correspond to an automorphism of the surface.

	A geodesic $c$ determines an analytic function $L(c)$ from $\mathcal{T}$ to $\mathbb{R}_{+}$, whose value at $x\in \mathcal{T}$ is given by the length of $c$ on the marked hyperbolic surface $S(x)$. This is a special case of a length function on $\mathcal{T}$.
	
	\begin{defn}[Length function]
		A finite set of curves $C=(c_{1}, \ldots, c_{k})$ together with a finite set of real, positive weights $A=(a_{1}, \ldots, a_{k})$ determine an analytic function $L(A,C):\mathcal{T}\rightarrow \mathbb{R}^{+}$ given by
		\begin{equation*}
			L(A,C)(x)\,=\, \sum_{j=1}^{k} a_{j}L(c_{j})(x)
		\end{equation*}
		$L(A,C)$ is a called a length function. 
	\end{defn}
	Length functions are known to be strictly convex with respect to the Weil-Petersson metric, \cite{Wolpert}. It was shown in \cite{SchmutzMorse} that a length function has a minimum in $\mathcal{T}$ iff $C$ fills. It follows from convexity that, when it exists, a minimum is unique. 
	
	\begin{defn}[Set of minima $\mathrm{Min}(C)$ \cite{SchmutzMorse}]
		Let $C$ be a finite set of curves. The set of minima, $\mathrm{Min}(C)$, is defined to be
		\begin{equation*}
			\{x\in \mathcal{T}\ |\ \exists \text{ an }A\in \mathbb{R}_{+}^{|C|}\text{ such that }L(A,C)(x)\leq L(A,C)(y) \ \forall y\in \mathcal{T}\}
		\end{equation*}
	\end{defn}


	\section{Filling sets of systoles}
	\label{sec:theory}
	
	The aim of this section is to use theoretical arguments to determine the set of filling systoles in the top-dimensional strata of the spine.
	
	\begin{defn}
		Let $p,q,r\in\{1,2,3,4,5,6\}$ be three distinct integers. A {\em 3-ring} is a set of three geodesics $\{p-q,q-r,r-p\}$.

If there is a $(1,1)$-subsurface $A$ of $S$ which contains the set $\{p-q,q-r,r-p\}$, then  $\{p-q,q-r,r-p\}$ is called a {\em triple}.
	\end{defn}
	
	\begin{lem} \label{lem-triple}
	Let $\{a,b,c\}$ be a 3-ring in $S$. If $a,b,c\in \mathrm{sys}(S)$, then $\{a,b,c\}$ is a triple.
	\end{lem}
	\begin{proof}Note that there is a unique separating geodesic $z$ in $S$ such that $a$ and $b$ are contained in the $(1,1)$-subsurface $A$ of $S$ with boundary $z$. For the geodesic $c$, there are two different topological possibilities; $c$ is either in $A$ or $c$ intersects $z$ four times. In the latter case,  $c\,\cap\, A$ has two disjoint components. It suffices to show that this case is impossible when $a,b,c$ are systoles of $S$. 
		
		\begin{figure}[!thpb]
			\centering
			\includegraphics[width=0.7\textwidth]{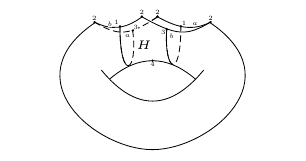}
			\caption{The hexagon $H$ from Lemma \ref{lem-triple} contained in $S\setminus\{a,b\}$. }
			\label{hexagon}
		\end{figure}
		
		Suppose $c\,\cap\, A$ has two disjoint components. Then $S\setminus\{a,b,c\}$ has a component $H$ which is a hexagon as shown in Figure \ref{hexagon}. The vertices of $H$ are fixed points of the hyperelliptic involution $\tau$ of $S$, and the edges are arcs of $a$, $b$, $c$, each with the same length given by half the systole length. The vertices will be labeled $1,2,3$ where each of $1,2,3$ represents a fixed point corresponding to two different vertices of $H$. One further fixed point of $\tau$, call it 4, lies in the center of $H$ and $\tau$ induces an involution of $H$. There are also three geodesic $\alpha=4-1$, $\beta=4-2$, $\gamma=4-3$ in $S$ which lie in $H$. These geodesics partition $H$ into six triangles. At least one of these triangles, call it $T$, has an inner angle $\geq \pi/3$ associated to the vertex 4 while the opposite side to this angle has length given by half the length of one of $a,b$ or $c$. Since the sum of the inner angles of $T$ is less than $\pi$ and since the shortest side of $T$ is opposite the smallest inner angle of $T$ a contradiction is obtained. The side of length equal to half of the length of $a$, $b$, or $c$ cannot be the shortest side of $T$.
	\end{proof}
	
	\begin{rem} {\em When we speak of triangles, of hexagons and so on, this is always with respect to the hyperbolic metric.}
	\end{rem}
	
	\begin{prop} \label{prop-star} Let $S$ be such that $\mathrm{sys}(S)$ contains (at least) two elements $a$ and $b$ that intersect. Then the smaller angle between $a$ and $b$ is $\geq \pi/4$. Moreover, equality only holds if $S$ is the Bolza surface.
	\end{prop}
	\begin{proof}
		We may assume that $a=1-2$ and $b=2-3$. Let $c=1-3$ be a geodesic in $S$ such that $\{a,b,c\}$ is a triple. By Lemma \ref{lem-triple} the triple $\{a,b,c\}$ is contained in a $(1,1)$-subsurface $A$ of $S$. Then
		$A\setminus \{a,b,c\}$ contains two connected components given by isometric  triangles; let $T$ be one of them. The vertices of these triangles are the three fixed points 1,2,3 of the hyperelliptic involution of $S$ while the sides are geodesic arcs contained in $a,b$ and $c$. Let $2x$ be the length of $a$ and of $b$ and let $2y$ be the length of $c$. Then the lengths of the sides of $T$ are given by $x$, $x$ and $y$. Let $\gamma$ be the inner angle in $T$ between $a$ and $b$, let $\beta$ and $\gamma$ be the two other inner angles of $T$. Since $a$ and $b$ are systoles, we have $\alpha=\beta\leq\gamma$.
		
		By Theorem 5.2 in \cite{SchmutzMaxima}, we have $\cosh(x)\leq 1+\sqrt{2}$, and if $\cosh(x)= 1+\sqrt{2}$ all three  inner angles of $T$ are $\pi/4$ and $S$ is the Bolza surface.
		
		Let $x<1+\sqrt{2}$ and $\gamma=\pi/4$. Then $y<x$, a contradiction. It follows that $\gamma>\pi/4$, concluding the proof. 
	\end{proof}
	
	\begin{defn} Let $a_{1},\ldots,a_{k}$ be $k$ different geodesics in a surface of genus 2 that all intersect in the same point, and $3\leq k\leq 5$. We say that $\{a_{1},\ldots,a_{k}\}$ is a $k$\textit{-star}.
	\end{defn}
	
	\begin{cor} \label{cor-star} Let $S$ be such that $\mathrm{sys}(S)$ contains a $k$-star, $3\leq k\leq 5$. Then $k=5$ is impossible and $k=4$ implies that $S$ is the Bolza surface.
	\end{cor}
	\begin{proof} If $k=5$, one of the angles between $a_{i}$ and $a_{i+1}$, $i=1,...,5$, where $i$ is taken modulo 5, is strictly smaller than $\pi/4$, contradicting Proposition \ref{prop-star}. If $k=4$, then all four angles between $a_{i}$ and $a_{i+1}$, $i=1,...,k-1$, where $i$ is taken modulo 4, are $\pi/4$ and $S$ is the Bolza surface, see Proposition \ref{prop-star}. 
	\end{proof}
	
	\begin{prop} \label{prop-triple} Let $S$ be such that $\mathrm{sys}(S)$ fills and contains a triple. Then $\mathrm{sys}(S)$ has 9 or 12 elements. In the latter case, $S$ is the Bolza surface and the systolic graph is given by the 1-skeleton of an octahedron. If $\mathrm{sys}(S)$ has 9 elements, the systolic graph is given by the 1-skeleton of a triangular prism.
	\end{prop}
	
	\begin{proof} 
		The proof will be divided into three parts, labelled parts 1, 2 and 3 for ease of reference.
		
		1) Let $\{a,b,c\}$ be a triple in $\mathrm{sys}(S)$. By Lemma \ref{lem-triple}, there is a unique separating geodesic $z$ in $S$ that cuts $S$ into two $(1,1)$-subsurfaces $A$ and $B$ such that $a,b,c$ lie in $A$. It follows from Section 4 of \cite{SchmutzMaxima} that $A$ and $B$ are isometric so that $B$ also  contains  a triple of systoles $\{a', b', c'\}$. The idea is that a triple of systoles in $A$ represents the global maximum of $syst$ for the surface $A$ analogous to the Bolza surface for the genus 2 surface. Then $A$ and $B$ must be isometric because otherwise there would be a geodesic in $B$ shorter than the systoles.
		
		Since $\mathrm{sys}(S)$ fills, there must also be a systole of $S$ which intersects $z$ so that $\mathrm{sys}(S)$ has at least 7 systoles.
		
		2) To each geodesic $p$ in $A$, there is a unique geodesic segment $t(p)$ in $A$ disjoint from  $p$ with endpoints on $z$, making an angle of $\pi/2$ with $z$ at the intersection points. Moreover, it follows from hyperbolic trigonometry that if $p$ and $q$ are two geodesics in $A$ with $L(p)>L(q)$, then $L(t(p))>L(t(q))$ (where $L$ stands for "length"). Since $a$, $b$, $c$ are systoles in $A$, $t(a)$, $t(b)$, $t(c)$ all have the same length and are strictly shorter than $t(p)$ whenever $p\not\in\{a,b,c\}$ is a geodesic in $A$. The same holds for $B$.
		
		\begin{rem}
			\label{passthrough}
			Note that, by symmetry, $t(a)$, $t(b)$ and $t(c)$ each pass through a Weierstrass point.
		\end{rem}
		
		Let $S_{1}$ be the surface of genus 2 obtained from $S$ by a twist deformation along $z$ such that $\bar{a}=t(a)\,\cup\,t(a')$ forms a geodesic in $S_{1}$. Then $\bar{b}=t(b)\,\cup\,t(b')$ and $\bar{c}=t(c)\,\cup\,t(c')$ also form geodesics in $S_{1}$ (where we may exchange the notation of $b'$ and $c'$, if necessary). By construction, $\bar{a}$, $\bar{b}$, $\bar{c}$ all have the same length and are the shortest geodesics of $S_{1}$ that intersect $z$. Note that $\bar{a}$, $\bar{b}$, and $\bar{c}$ are pairwise disjoint.
		
		If $\bar{a}$, $\bar{b}$, and $\bar{c}$ are systoles of $S_{1}$, the length of $t(a)$ must be half the length of $a$. Denote by $2x$ the length of $a$ and $2y$ the length of $t(a)$. Then 
		\begin{equation}\label{eq-1}
			\sinh(y)\sinh(\zeta/2)=\cosh(x)
		\end{equation}
		where $2\zeta$ is the length of $z$. By Corollary 4.1 in \cite{SchmutzMaxima}, we have 
		\begin{equation}\label{eq-2}
			4\cosh^{3}x-6\cosh{2}(x)+1=\cosh(\zeta).
		\end{equation}
		It follows from Equations  (\ref{eq-1}) and (\ref{eq-2}) that $y$ decreases if $x$ increases. Moreover, if $x\rightarrow \infty$, then $y\rightarrow 0$. On the other hand, if $\zeta\rightarrow 0$, then $y\rightarrow\infty$. Hence there is a unique $x$ such that $y=2x$. We call the corresponding surface $S_{2}$. Note that, in $S_{2}$, through every fixed point of the hyperelliptic involution $\tau$, exactly three of the nine systoles pass.
		
		3) By construction in part 2, $\mathrm{sys}(S_{2})$ has exactly 9 elements with systolic graph given by the 1-skeleton of a triangular prism. By part 2 of the argument, we may increase the length of $x$ in $S_{2}$, thereby decreasing the length of $y$. This can be compensated by a twist deformation along $z$. We thereby obtain a one-parameter family of surfaces, determined by $x$, which all have exactly 9 systoles. 
		This can be continued as long as the number of systoles does not increase. If there is a further systole, we obtain a 4-star of systoles, hence the corresponding surface is the Bolza surface by Corollary \ref{cor-star}.  
	\end{proof}
	
	\begin{defn} We call $M_{9}$ the surface $S_{2}$ constructed during the proof of Proposition \ref{prop-triple}. 
	\end{defn}
	It was shown in \cite{SchmutzMorse}, Theorem 44 that $M_{9}$ is a critical point of $syst$ of index 5.
	
	\begin{cor} \label{cor-triple}  This corollary has two parts.
		\begin{enumerate}
			\item{Let $2x$ be the length of a systole in the Bolza surface. Then $\cosh(x)=1+\sqrt{2}\,\sim 2.41$.}
			\item{Let $2x$ the length of a systole in the surface $M_{9}$. Then $\displaystyle \cosh(x)=\,\frac{5+\sqrt{17}}{4}\,\sim 2.28$.}
		\end{enumerate}
	\end{cor}
	\begin{proof}Part (1) of the corollary is well-known, see also the proof of Proposition \ref{prop-star} above. The second part is a consequence of Equations (\ref{eq-1}) and  (\ref{eq-2}). 
	\end{proof}
	
	Denote by $i(a,b)$ the number of intersections between the geodesics $a$ and $b$.
	
	\begin{defn}
		Let  $a,b,a'$ and $b'$ be four geodesics such that $i(a,b)=i(a',b')=1$ and $i(a,a')=i(a,b')=i(b,a')=i(b,b')=0$. We say that $\{a,b\}$,$\{a',b'\}$ are disjoint pairs.
	\end{defn}
	\begin{defn}
		Let $a_{1},\ldots,a_{k}$ be $k$ distinct geodesics, $3\leq k\leq 5$, such that $a_{i}$ intersects (once) only $a_{i-1}$ and $a_{i+1}$, $i=2,\ldots,k-1$, while $a_{1}$ intersects only $a_{2}$, $a_{k}$ intersects only $a_{k-1}$. We then say that $\{a_{1},\ldots,a_{k}\}$ is a $k$-chain.
	\end{defn}
	\begin{defn}
		Let $a_{1},\ldots,a_{k}$ be $k$ distinct geodesics, $4\leq k\leq 6$, such that $\{a_{1},\ldots,a_{k-1}\}$ is a $k-1$-chain while $a_{k}$ intersects only $a_{k-1}$ and $a_{1}$. We say that $\{a_{1},\ldots,a_{k}\}$ is a $k$-ring.
	\end{defn}
	
	\begin{lem} \label{lem-4ring} A 4-ring in $S$ either fills or is contained in a unique $(1,2)$-subsurface of $S$.
	\end{lem}
	
	\begin{proof}
		Let $\{a,b,c\}$ be a 3-chain of geodesics in $S$. Cut $S$ along $a$ and $c$, the result is a $(0,4)$-subsurface $D$ with boundary geodesics $a_{1}$, $a_{2}$, $c_{1}$, $c_{2}$. As $b$ intersects $a$ and $c$, each in exactly one point, $b\cap D$ has two connected components; one (call it $b_{1}$) with endpoints on $a_{1}$ and $c_{1}$ and the other, call it $b_{2}$ with endpoints on $a_{2}$ and $c_{2}$. Then as shown in Figure \ref{Tshirt} there is a geodesic in $D$ separating $b_{1}\cup a_{1}\cup c_{1}$ and $b_{2}\cup a_{2}\cup c_{2}$.  We therefore have shown that there is a unique $(1,2)$-subsurface $C$ of $S$ which contains $\{a,b,c\}$. 
		
		\begin{figure}[!thpb]
			\centering
			\includegraphics[width=0.4\textwidth]{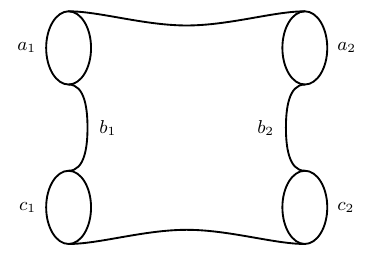}
			\caption{The subsurface $D$ from the proof of Lemma \ref{lem-4ring}. }
			\label{Tshirt}
		\end{figure}
		
		\begin{figure}[!thpb]
			\centering
			\includegraphics[width=0.9\textwidth]{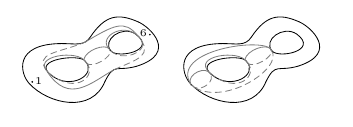}
			\caption{On the left is a filling 4-ring and on the right a 4-ring that does not fill.}
			\label{4rings}
		\end{figure}
		
		Let now $\{a,b,c,d\}$ be a 4-ring of geodesics in $S$. If $d$ is in $C$ (which is possible), then we have the second case of the lemma. If $d$ intersects the boundary geodesics of $C$ 
		(which is also possible), then $\{a,b,c,d\}$ fills $S$. These two possibilities are illustrated in Figure \ref{4rings}.
	\end{proof}
	
	\begin{prop} \label{prop-4ring} 
		Let $S$ be such that $\mathrm{sys}(S)$  contains a 4-ring that fills. Then $S$ is the Bolza surface and $\mathrm{sys}(S)$ has 12 elements.
	\end{prop}
	\begin{proof}Let $\{a,b,c,d\}\in \mathrm{sys}(S)$ be a 4-ring in $S$ that fills. There are then 2 Weierstrass points $p_{1}$ and $p_{2}$ not contained in any of the geodesics $\{a,b,c,d\}$. Since $\{a,b,c,d\}$ fill, $p_{1}$ and $p_{2}$ must be in two different connected components of $S\setminus \{a,b,c,d\}$. Recall that all crossings of the geodesics  $\{a,b,c,d\}$ occur at the Weierstrass points, so the vertices of the polygons $S\setminus \{a,b,c,d\}$ are all on the Weierstrass points. Cutting $S$ along the 3-chain given by $\{a,b,c\}$ gives an annulus $A$ with $p_{1}$ and $p_{2}$ in the interior, and for which $A\cap d$ has two connected components, and cutting $A$ along $d$ gives two isometric octagons with all sides of length $x$. This is shown in the left side of Figure \ref{4rings} with $p_{1}=1$ and $p_{2}=6$.
		
		Let $E$ be one of the octagons of $S\setminus \{a,b,c,d\}$. The hyperelliptic involution $\tau$ induces an involution of $E$ and fixes the center of $E$ which we call 1. Two opposite vertices of $E$ represent the same fixed point of $\tau$; we assume that they represent $2,3,4,5$. The geodesic segments  in $E$ between two opposite vertices are  geodesics $1-2$, $1-3$, $1-4$, and $1-5$ in $S$. They partition $E$ into 8 triangles. Since the sides of $E$ correspond to systoles in $S$, in each of these 8 triangles, the smallest inner angle is that in 1, in other words that which is opposite to a side of $E$. It follows that the sum of all inner angles of all 8 triangles is at least $6\pi$. Therefore, the sum of the 8 inner angles of $E$ is at least $4\pi$. But since the area of $E$ is half of the area of $S$, the sum of the inner angles of $E$ is exactly $2\pi$. This implies that each of the 8 triangles in $E$ is equilateral. Therefore, the geodesics $1-2$, $1-3$, $1-4$, $1-5$ are also systoles and $S$ is the Bolza surface by Proposition \ref{prop-triple}. 
	\end{proof}
	
	\begin{lem} \label{lem-pair} Suppose $\{a,b\}$, $\{a',b'\}$ are disjoint pairs of geodesics in $\mathrm{sys}(S)$. Then there is a separating geodesic $z$ in $S$ that separates $S$ into two $(1,1)$-subsurfaces $A$ and $A'$, where $A$ contains $a$ and $b$ and $A'$ contains $a'$ and $b'$. The two subsurfaces are isometric and there are involutions $\psi$ and $\psi'=\psi\circ \tau$ of $S$ that exchange the two subsurfaces. 
	\end{lem}
	\begin{proof}
		We first show the existence of $z$, $A$ and $A'$. The geodesic $z$ is on the boundary of the subsurface filled by $a$ and $b$ or, equivalently $a'$ and $b'$. Cutting $S$ along $z$, $a$ and $a'$ gives two isometric pants. The lengths of $b$ and $b'$ are only equal in $S$ if the twist parameters around $a$ and $a'$ are equal up to sign. This gives the isometric surfaces $A$ and $A'$. 
		
		Denote by $2x$ the length of $a$ and by $2\zeta$ the length of $z$.  Cut $A$ along $a$; the result is a pant that can be decomposed into two isometric copies of a right-angled hexagon $H$ in the usual way. The hexagon $H$ is determined by the edge lengths $x$ and $\zeta$. Similarly when we cut $A'$ along $a'$. Let $s$ be the side of $H$ opposite the side on the geodesic $z$. One fixed point of the hyperelliptic involution $\tau$ of $S$ lies in the center of $s$, call it 1. Let 2 be the fixed point of $\tau$ that is the intersection point of $a$ and $b$. Then 2 lies on one of the sides of $H$ on $a$. Again, we have the same situation with respect to $A'$. Therefore, there is a geodesic $c$ in $A$ and a geodesic $c'$ in $A'$ such that $\{a,b,c\}$ and $\{a',b',c'\}$ are triples where $c$ and $c'$ have the same length $2y$. This is illustrated in Figure \ref{Hexfig}.
		
		\begin{figure}[!thpb]
			\centering
			\includegraphics[width=0.5\textwidth]{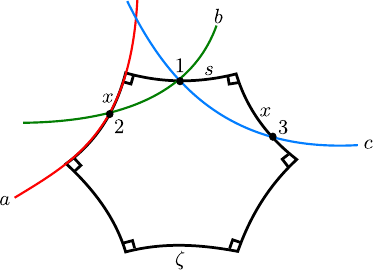}
			\caption{ The hexagon $H$ from the proof of Lemma \ref{lem-pair} with the geodesics $a$ (in red), $b$ (in green) and $c$ (in blue) in the universal cover.}
			\label{Hexfig}
		\end{figure}
		
		Now cut $A$ along $c$. Again, the result is a pair of pants consisting of two isometric copies of a right-angled hexagon $K$. We call $y$ the two sides of $K$ lying along the geodesic $c$ and $\zeta$ the side of $K$ lying along $z$. Let $t$ be the side of $K$ opposite the side $\zeta$. For ease of notation we will use the same symbol for an edge as for its length. The fixed point 1 lies on the center of $t$ while the fixed points 2 and 3 of $\tau$ lie on the center of the sides $y$.
		The two remaining sides of $K$ have the same length and will be labelled $p$. 
		
		Cutting $A'$ along $c'$ produces the same situation with a right-angled hexagon $K'$ with two sides $c'$, two sides $p'$, one side $t'$ and on side $\zeta'$, each with the same length as the corresponding side of $K$. Moreover, the fixed points 4,5,6 of $\tau$ lie on the corresponding places; more precisely, 4 lies on the center of $t'$.
		
		Denote by $p_{1}$, $p_{2}$ the two vertices of $K$ that are endpoints of $\zeta$ and denote by $p'_{1}$, $p'_{2}$ the corresponding vertices of $K'$. Then $p_{1}$, $p_{2}$, $p'_{1}$, $p'_{2}$ lie on $z$ in $S$. Let $f_{i}$ be the midpoints on $z$ of the arc connecting $p_{i}$ and $p'_{i}$, $i=1,2$. It follows by construction that there is an involution $\psi$ of $S$ with fixed points $f_{1}$ and $f_{2}$. Moreover, there is an involution $\psi'=\psi\circ\tau$ of $S$ which has the two fixed points $f'_{1}$, $f'_{2}$ on $z$, corresponding to the centers on $z$ between $p_{1}$ and $p_{2}'$ and between $p_{2}$ and $p_{1}'$ respectively. This is illustrated in Figure \ref{involution}. By construction, both $\psi$ and $\psi'$ exchange 1 and 4.  We fix the notation such that $\psi$ and $\psi'$ exchange 2 and 6 as well as 3 and 5. 
	\end{proof}
	
	\begin{figure}[!thpb]
		\centering
		\includegraphics[width=0.4\textwidth]{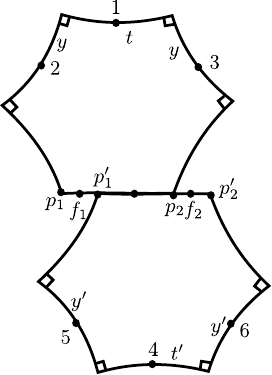}
		\caption{The involution $\psi'$ from Lemma \ref{lem-pair} fixes the points $f'_{1}$ and $f'_{2}$, and fixes the Weierstrass points setwise, where $\psi'(1)=4$, $\psi'(2)=6$ and $\psi'(3)=5$. Moreover, $\psi'$ maps geodesic segments in $K$ to the corresponding geodesic segments in $K'$. For example, the geodesic segment between Weierstrass point number 2 and $f'_{1}$ is mapped to the geodesic segment in $K'$ between Weierstrass point number 6 and $f'_{1}$. }
		\label{involution}
	\end{figure}
	
	\begin{prop} \label{prop-pair} Let $S$ be such that $\mathrm{sys}(S)$ fills and contains disjoint pairs $\{a,b\}$, $\{a',b'\}$, but not a triple. Then one of the following holds
		\begin{enumerate}
			\item{$\mathrm{sys}(S)$ is a 5-chain,}
			\item{$\mathrm{sys}(S)$ is a 6-ring,}
			\item{$\mathrm{sys}(S)$ has exactly 6 elements and contains a non-filling 4-ring.}
		\end{enumerate}
	\end{prop}
	\begin{proof} We will use the same notation as in Lemma \ref{lem-pair} and its proof. The argument will be subdivided into three parts for ease of reference; these do not correspond to the three statements of the proposition.
		
		1) By Lemma \ref{lem-pair}, $S$ has two isometric $(1,1)$-subsurfaces $A$ and $A'$, with common boundary geodesic $z$, where $A$ contains $a$ and $b$ and $A'$ contains $a'$ and $b'$. Choose $c$ in $A$ such that $\{a,b,c\}$ form a triple and such that the length of $c$ is minimal subject to this constraint. By hypothesis,  $c\not\in \mathrm{sys}(S)$ so that the length $2y$ of $c$ is strictly larger than the length $2x$ of $a$. It then follows that $t(c)$ is strictly longer than $t(a)$ (and $t(b)$), where $t(c)$, $t(a)$, $t(b)$ are defined as in the proof of Proposition \ref{prop-triple}. 
		
		Since $\mathrm{sys}(S)$ fills, there is an element $e\in \mathrm{sys}(S)$ that intersects $z$. Suppose $e=1-4$. Let $d\in\{a,b,c\}$ and denote by $T_{d_{i}}$ the pair of endpoints of $t_{d}$ on $z$, $i=1,2$. Define $T_{d'_{i}}$ analogously with $d'\in\{a',b',c'\}$ in place of $d$. Since $A$ and $A'$ are isometric, the shorter distance between $T_{d_{i}}$ and $T_{d'_{j}}$ on $z$, $i,j\in\{1,2\}$, is the same for $d=a$, $d=b$, and $d=c$. Since $t_{c}$ is strictly longer than $t_{a}$, it follows  that there is a geodesic in $S$ that intersects $z$ disjoint from and shorter than $1-4$. Hence $e=1-4$ is not possible. 
		
		2) Now assume that $e=2-5$ so that $\psi(e)=3-6\neq e$. 
		
		Then  $\mathrm{sys}(S)$ contains the 6-ring $\{1-2, 2-5, 5-4, 4-6, 6-3, 3-1\}$. Suppose that $\mathrm{sys}(S)$ contains a further element $e'$. Since $\mathrm{sys}(S)$ does not contain a triple, we have only the possibilities $e'=1-4$, $e'=2-6$, $e'=3-5$. We have seen in part (1) that $e'=1-4$ is not possible. The same argument excludes $e'=2-6$ as well as $e'=3-5$ since $\mathrm{sys}(S)$ contains the disjoint pairs $\{1-2,2-5\}$, $\{3-6,4-6\}$ as well as the disjoint pairs $\{1-3,3-6\}$, $\{2-5,4-5\}$. Therefore, $\mathrm{sys}(S)$ is a 6-ring.
		
		If $e=3-6$ and $\psi(e)=2-5$, we obtain the same result.
		
		Assume now that $e=2-6$ so that $\psi(e)=e$. Then $\mathrm{sys}(S)$ contains the 5-chain $\{3-1, 1-2, 2-6, 6-4, 4-5\}$. Suppose $\mathrm{sys}(S)$ contains a further element $e'$. If $e'=3-5$, we are in the previous case where $\mathrm{sys}(S)$ is a 6-ring. Let $e'=3-6$ so that $\psi(e')=2-5$. Then $\mathrm{sys}(S)$ contains the disjoint pairs $\{1-2,2-5\}$, $\{4-3,4-5\}$, but since $\mathrm{sys}(S)$ also contains $1-4$, we obtain a contradiction by part (1). Let $e'=3-5$ so that $\psi(e')=2-6$. Then $\mathrm{sys}(S)$ contains the disjoint pairs $\{1-2,1-4\}$, $\{3-6,4-6\}$ which gives a contradiction since by (1), $2-6$  then cannot be a systole. That $e'=3-4$, $\psi(e')=1-5$ is not possible will be shown in part (3). Since $\mathrm{sys}(S)$ does not contain a triple and since $e'=1-4$ is excluded by part (1), the existence of $e'$ is not possible. Therefore, we have one of the first two cases of the proposition.
		
		The same result is obtained if  $e=3-5$. 
		
		3) Assume now that $e=1-5$ so that  $\psi(e)=3-4$. Consequently, $\mathrm{sys}(S)$ contains the 4-ring $\{1-3,3-4,4-5,5-1\}$. Suppose that $\mathrm{sys}(S)$ contains a further element $e'$. Since $\mathrm{sys}(S)$ does not contain a triple, nor a 4-star (by Corollary \ref{cor-star}), the unique possibility is $e'=2-6$, but by part (2), this is not possible. Therefore, $\mathrm{sys}(S)$ has exactly 6 elements and we have the third case of the proposition.
		
		The same result is obtained if $e=3-4$ or $e=1-6$ or $e=2-4$. 
	\end{proof}
	
	\begin{cor} \label{cor-pair1} Let $S$ be such that $\mathrm{sys}(S)$ contains a 5-ring. Then $\mathrm{sys}(S)$ has exactly five elements or contains a triple.
	\end{cor}
	\begin{proof}
		By hypothesis, $\mathrm{sys}(S)$ contains (at least) five elements which may be described as $1-2$, $2-3$, $3-4$, $4-5$, $5-1$. Assume that $\mathrm{sys}(S)$ contains a further element $a$ which is $1-q$. If $q=3$ or $q=4$, then $\mathrm{sys}(S)$ contains a triple. If $q=6$,  then $\{1-5, 1-6\}$, $\{3-2,3-4\}$ are disjoint pairs and  $\mathrm{sys}(S)$ contains a triple by Proposition \ref{prop-pair}. If $a=p-q$ for $p\in\{2,3,4,5\}$, the same argument applies, thereby proving the corollary. 
	\end{proof}
	
	\begin{cor} \label{cor-pair2} Let $S$ be such that $\mathrm{sys}(S)$ fills and contains a 4-ring that does not fill. Then $\mathrm{sys}(S)$ either contains a triple or has exactly 6 elements.
	\end{cor}
	\begin{proof}
		Let $\{a,b,c,d\}\in \mathrm{sys}(S)$ be a 4-ring in $S$ that does not fill. By Lemma \ref{lem-4ring}, there is a $(1,2)$-subsurface $C$ of $S$ containing $\{a,b,c,d\}$. Since the two boundary geodesics of $C$ have the same length and $a,b,c,d$ all have the same length, $C$ is a union of two isometric pairs of pants, each with a boundary geodesic on $a$ and on $c$. Alternatively, $C$ is a union of two isometric pairs of pants, each with a boundary geodesic on $b$ and on $d$. There is an involution $\psi$ of $C$ obtained by composing two reflections. One of these reflections restricts to a reflection of each of the pants with boundaries containing $a$ and $c$, interchanging $a$ and $c$ and fixing $b$ and $d$ setwise. The other reflection restricts to a reflection of each of the pants with boundaries containing $b$ and $d$, interchanging $b$ and $d$ and fixing $a$ and $c$ setwise. One of these reflections is illustrated in Figure \ref{involution2}.
		
		\begin{figure}[!thpb]
			\centering
			\includegraphics[width=0.3\textwidth]{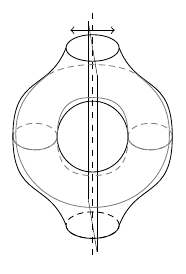}
			\caption{The geodesics $\{a,b,c,d\}$ inside the $(1,2)$-subsurface $C$. The reflection shown interchanges $a$ and $c$ and fixes $b$ and $d$ setwise.}
			\label{involution2}
		\end{figure}
		
		The hyperelliptic involution $\tau$ of $S$ also induces an involution of $C$. 
		
		By hypothesis, there is an element $e\in \mathrm{sys}(S)$ that intersects the boundary geodesic of $C$. Suppose that $e\,\cap\,C$ has two connected components. Then $e$ intersects all elements of $\{a,b,c,d\}$. This implies that $\{a,b,e\}$ is a triple in $S$ that is not contained in a $(1,1)$-subsurface $S$, contradicting Lemma \ref{lem-triple}. Therefore, $e\,\cap\,C$ has a unique component, also called $e$, and $\psi(e)\neq e$. We may assume that $e$ intersects $a$ and $b$. Then $e'$ intersects $c$ and $d$ so that $\{e,a\}$, $\{e',c\}$ are disjoint pairs. The corollary now follows from Proposition \ref{prop-pair}. 
	\end{proof}

	\begin{defn}  We denote by $M_{6}$ an unmarked surface $S$ representing a point in moduli space with the property that $\mathrm{sys}(S)$ is a 6-ring and all intersections between systoles occur at right angles.
	\end{defn}
	\begin{defn} We call $M_{5}$ an unmarked surface $S$ representing a point in moduli space with the property that $\mathrm{sys}(S)$ is a 5-ring for which all intersections between systoles occur with the same angle.
	\end{defn}
	The surfaces  $M_{6}$ and $M_{5}$ were defined in Theorem 44 of \cite{SchmutzMorse}, where they were shown to be critical points of $syst$ of index 3 and 4 respectively.
	
	\begin{prop} \label{prop-65}
		This proposition has two parts. 
		\begin{enumerate}
			\item{ $M_{6}$ exists and is unique. If $2x$ is the length of its systoles, then $\cosh(x)=2$ and this is the minimal length for all surfaces $S$ where $\mathrm{sys}(S)$ contains a 6-ring.}
			\item{$M_{5}$ exists and is unique. If $2x$ is the length of its systoles, then $\displaystyle\cosh(x)=\,\frac{2+\sqrt{5}}{2}\,\sim 2.12$ and this the minimal length for all surfaces $S$ where $\mathrm{sys}(S)$ contains a 5-ring.}
		\end{enumerate}
	\end{prop}
	\begin{proof}
		To prove part (1), note that the existence of $M_{6}$ and of $M_{5}$ as well as their uniqueness was proven in Section 5 of \cite{SchmutzMorse}. In $M_{6}$, the 6 systoles determine a tessellation by four isometric right-angled hexagons of edge length $x$. Consequently, $\cosh(x)=2$. The surface $M_{5}$ can be partitioned into 10 isometric quadrilaterals. Two opposite angles of such a quadrilateral $Q$ are $\pi/2$ while the other two are $\pi/5$ and $\pi/10$.  The diagonal in $Q$ between the two right angles has the length $x$. This gives $\displaystyle\cosh(x)=\,\frac{2+\sqrt{5}}{2}$ as required.
		The minimality of the length of the systoles in $M_{6}$ and in $M_{5}$ follows from part (iii) and part (iv) of the proof of Theorem 44 in \cite{SchmutzMorse}. 
	\end{proof}
	
	\begin{defn}  Suppose $a,b,c$ and $d$ are four different geodesics in $S$ such that $\{a,b,c\}$ is a 3-chain and $\{b,c,d\}$ is a 3-star. Suppose also that $a$ and $d$ are disjoint. We then say that $\{a,b,c,d\}$ is a long $Y$.
	\end{defn}
	
	\begin{defn} We say that a filling set $C$ of geodesics is minimal if for every $a\in C$, $C\setminus\{a\}$ does not fill.
	\end{defn} 
	
	\begin{prop}[Proposition \label{prop-minimal} from the introduction]
Every minimal filling subset of systoles is either a 4-star, a 4-ring, a 4-chain, or a long $Y$.
	\end{prop}
	\begin{proof}
		Let $\Sigma$ be a subset of $\mathrm{sys}(S)$ that fills and is minimal. Recall that the systolic graph corresponding to $\Sigma$ must be connected. If $\Sigma$ contained only 2 elements, there would be a separating geodesic in $S$ disjoint from all the systoles. If $\Sigma$ contained only 3 elements, then it is a triple, a 3-star, or a 3-chain. In all cases, there is a geodesic in $S$ disjoint from all the elements of $\Sigma$, see Lemma \ref{lem-triple} and its proof as well as the proof of Lemma \ref{lem-4ring}. This shows that $\Sigma$ has at least 4 elements. Moreover, $\Sigma$ cannot contain a triple since $\Sigma\setminus\{a\}$ would also fill if $a$ is chosen to be an element of this triple. 
		
		Assume that $\Sigma$ has  4 elements and contains a 3-star. There is then a unique $(1,2)$-subsurface  of $S$ that contains this 3-star. Moreover, since $\Sigma$ does not contain a triple, the fourth element of $\Sigma$ necessarily intersects the boundary of this subsurface and hence $\Sigma$ fills. Note $\Sigma$ is a long $Y$.
		
		For what follows we may assume that $\Sigma$ is minimal and does not contain a 3-star. In particular, $\Sigma$ is not a 4-star nor a long $Y$. Moreover, $\Sigma$ cannot contain a 4-ring that does not fill since $\Sigma$ would then contain a 3-star (one of the elements of $\Sigma$ not contained in the 4-ring must intersect an element of the 4-ring). It therefore follows that every cardinality 3 subset $\Sigma'$  of $\Sigma$ with connected systolic graph is a 3-chain. If $\Sigma$ is a 4-chain, then $\Sigma$ fills by Lemma \ref{lem-4ring} and its proof. Therefore, we may exclude that $\Sigma$ contains a 4-chain.  It follows that  $\Sigma$ is a 4-ring that fills and we are done. 
	\end{proof}
	
	\begin{defn} Let $C$ be a set of geodesics that fills and is minimal. We denote by $E(C)$ the set of all surfaces of genus 2 where the elements of $C$ all have the same length.
	\end{defn}

\begin{lem} \label{lem-ring1} Let $S$ be a $(1,2)$-surface where the two boundary components $t_{1}$ and $t_{2}$ have the same length $2t$. Let $1$, $2$, $3$, $4$ be the fixed points of the hyperelliptic involution of $S$. Let $S$ contain a 4-ring $a=1-2$, $b=2-3$, $c=3-4$, $d=1-4$ such that these four geodesics have the same length $2x$. Then $S$ is determined by $x$ and $t$.
\end{lem}
\begin{proof}
The proof will be divided into two pieces. 

a) The geodesics $a, b, c$ and $d$ are each made up of two geodesics arcs (that will be called halves) with endpoints on the fixed points. Four of these halves contained in $a,b,c$ and $d$ form a (hyperbolic) rhombus. Therefore, the two diagonals in this rhombus $1-3$ of length $2e$ and $2-4$ of length $2f$ intersect perpendicularly. Moreover, $1-3$ halves the angle in 1 between $a$ and $d$.

b) $S$ contains the $(0,3)$-subsurface $B$ with boundary components $t_{1}$, $a$, $c$. $B$ is determined by $x$ and $t$. By a), the geodesic segments $e$ and $f$ intersect perpendicularly in $M$ where $M$ is the center of the common perpendicular $p$ in $B$ between $a$ and $c$. The length $L(p)$ of $p$ is determined by $x$ and $t$. Let $P$ be the intersection point between $p$ and $a$, compare Figure \ref{fig-ring1}.

By a), the angle in 1 between $a$ and $e$ is the same as the angle in 1 between $e$ and $d$, we call it $\gamma$ (compare Figure \ref{fig-ring1}). By hyperbolic trigonometry, we have 
$$\sin(\gamma) =\,\frac{\sinh(f)}{\sinh(x)}\,=\,\frac{\sinh(L(p)/2)}{\sinh(e)}\:.$$
Therefore, $\sinh(e)\sinh(f)$ is determined by $x$ and $t$. Since $\cosh(e)\cosh(f)=\cosh(x)$, $\cosh(e)\cosh(f)$ is also determined by $x$ and $t$. We conclude that $e+f$ as well as $e-f$ is determined by $x$ and $t$ from which the lemma follows. 
\end{proof}

\begin{figure} 
\centering
\includegraphics[width=0.4\textwidth]{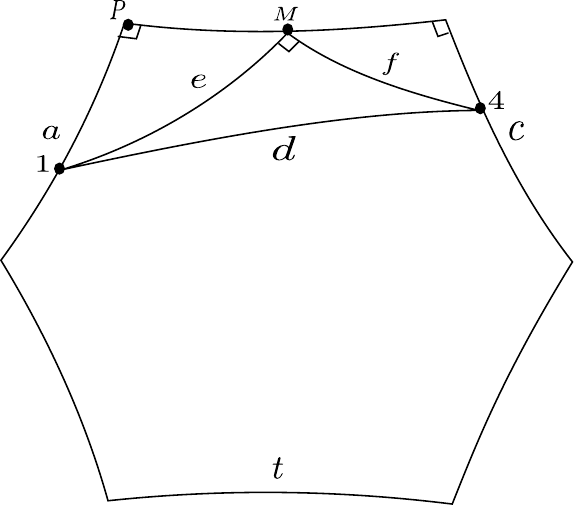}
\caption{The $(0,3)$-subsurface $B$ is composed by two isometric right-angled hexagons, one is showed in this figure.}
 \label{fig-ring1}
\end{figure}

\begin{figure} 
\centering
\includegraphics[width=0.4\textwidth]{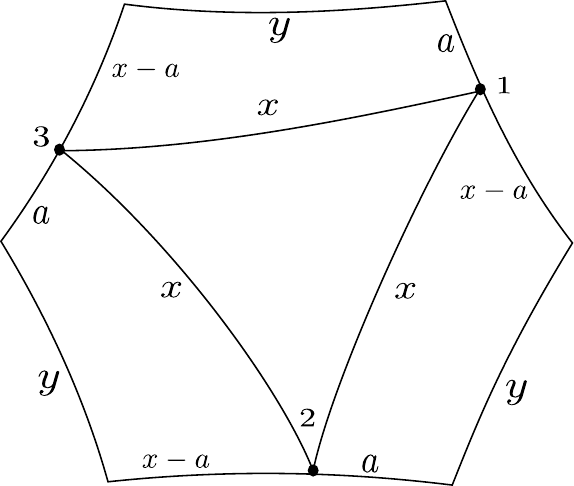}
\caption{The  right-angled hexagon $H$ of Corollary \ref{cor-ring1}.}
\label{fig-ring2}
\end{figure}

\begin{cor} \label{cor-ring1} Let $S\in \mathcal{T}$ be such that $S'=syst(S)$ is the 1-skeleton of a triangular prism. Let $2x$ be the length of these systoles. Let 
$$p,q,r\in S',\,p=1-4,\,q=2-5,\,r=3-6,$$ be
such that none of them belongs to a triple in $S'$. Let $B$ be a $(0,3)$-subsurface of $S$ with boundary geodesics $p,q,r$. Let $H$ be the right-angled hexagon contained in $B$ such that $1$, $2$, and $3$ lie in $H$ where $\{1-2,1-3,2-3\}$ is a triple in $S'$. Then there is an $a$, $0<a\leq x/2,$ such that $H$ is as in Figure \ref{fig-ring2}. 
\end{cor}
\begin{proof}
This follows from Lemma \ref{lem-ring1}. 
\end{proof}

\begin{prop} \label{prop-ring1} 
There is no $S\in\mathcal{T}$ such that $\mathrm{sys}(S)=\{1-2,2-3,3-4,4-5,2-5,4-6\}$. 
\end{prop}
\begin{proof}
The proof will be broken into four parts, labeled a, b, c and d for ease of reference.

a) Suppose that such an $S$ exists. Let $2x$ be the length of the systoles of $S$. Let $t=1-6$ be the unique simple closed geodesic that intersects none of $2-3,3-4,4-5,2-5$; let $2\tau$ be the length of $t$. Let $z$ be the separating closed geodesic that separates $1-2,1-6$ from $3-4,4-5$; let $2\zeta$ be the length of $z$. Let $p=2-4$, $q=3-5$ such that neither $p$ nor $q$ intersects $t$; recall from Lemma \ref{lem-ring1} and its proof that $p$ and $q$, as the diagonals of a hyperbolic rhombus, intersect perpendicularly.

b) Suppose that $1-2$ and $1-6$ intersect perpendicularly. This implies $\sinh(\tau)\sinh(x)=\cosh(\zeta/2)$. By a), $\{3-4,4-5,3-5\}$ form a triple and there is a set of halves contained in $3-4, 4-5$ and $3-5$ form a triangle $T$. In $T$, the angle between $3-4$ and $4-5$ is strictly smaller than $\pi/2$ by Proposition \ref{prop-star}. Therefore, the ``height'' $h$ in $T$ from $3$ to $4-5$ is strictly smaller than $x$. The symbol $h$ will also be used to denote the length of $h$. Then $\sinh(h)\sinh(x)=\cosh(\zeta/2)$. It follows that $h=\tau<x$, contradicting that the systoles of $S$ have length $2x$. This shows that the angle $\phi$ between $1-2$ and $1-6$ is bounded away from $\pi/2$. By the same argument, the angle between $5-6$ and $1-6$ is bounded away from $\pi/2$.

c) Assume $S$ is such that $p$ and $q$ have different lengths. By Lemma \ref{lem-ring1}, there then exist surfaces in the neighborhood of $S$ in $\mathcal{T}$ such that $2-3,3-4,4-5,2-5$ all keep the length $2x$ while $t$ gets strictly longer than $\tau$ and the difference of the lengths of $p$ and $q$ becomes strictly smaller. Moreover, by a twist deformation along $t$ and by b), $1-2$ and $5-6$ keep the same length $2x$. This can be continued until $p$ and $q$ have the same length or until we obtain a further systole. Note that during these changes, $t$ becomes strictly longer than $2\tau$ so  that $1-6$ cannot become a systole. Assume that we obtain a further systole $\sigma$ in $\Sigma$. Then $\sigma=2-4$ or $\sigma=3-5$ is impossible by Proposition \ref{prop-star}. If $\sigma= 1-3$, by the symmetry of $\Sigma$, $4-6$ is also a systole and, by Proposition \ref{prop-triple}, $1-6$ is a systole which was excluded. By the same argument $\sigma\in\{1-3,3-6,5-6\}$ is impossible. Finally, $\sigma\in \{1-4,2-6\}$ is impossible by Proposition \ref{prop-star}, considering that four systoles intersect in $4$ or in $2$.

d) By c) we may assume that $p$ and $q$ have the same length in $S$. Consider the $(0,3)$-subsurface $B$ of $S$ with boundary geodesics $1-6$, $2-3$ and $4-5$ and let $H$ be the right-angled hexagon contained in $B$ such that $2$ and $5$ lie on $H$. In $H$, let $C_{16}$, $C_{23}$, $C_{45}$ be the centers of the sides  corresponding to $1-6$, $2-3$,$4-5$, respectively. Since $p$ and $q$ have the same length, it follows from Lemma \ref{lem-ring1} that $2=C_{23}$, $5=C_{45}$. Moreover, $2-3$ must be strictly shorter than the geodesic segment $\sigma_{12}$ between $2=C_{23}$ and $C_{16}$. 

\begin{figure} 
\centering
\includegraphics[width=0.4\textwidth]{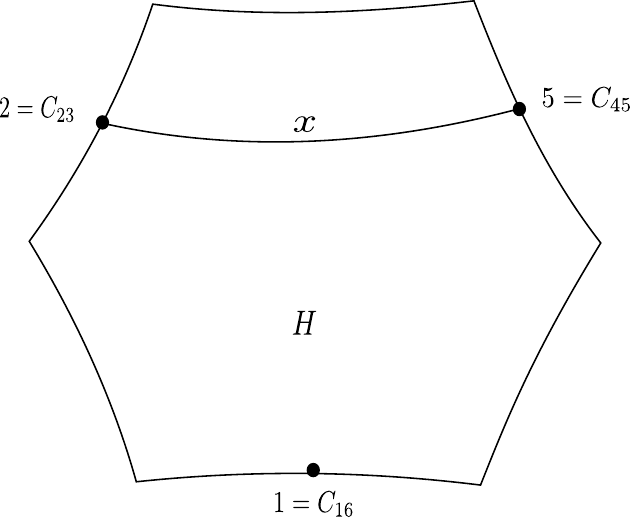}
\caption{Figure from part d) of the proof of Proposition \ref{prop-ring1}.}
\label{fig-ring5}
\end{figure}

Suppose now that $\Sigma$ is the surface obtained by a twist deformation of $S$ along $1-6$ such that $C_{16}$ is the point $1$, see Figure \ref{fig-ring5}. Since the geodesic segment $\sigma_{15}$ between $C_{45}$ and $C_{16}$ has the same length as $\sigma_{12}$, we have simple closed geodesics $1-2$ and $1-5$ in $\Sigma$ of the same length. By the symmetry of $\Sigma$, there are also simple closed geodesics $3-6$ and $4-6$ of the same length as $1-2$ and $1-5$. It then follows by the argument in part 1) of the proof of Proposition \ref{prop-pair} that $1-6$ is strictly shorter than $2-3$ and $4-5$. Since the lengths of $1-6$, $2-3$, and $4-5$ are not changed by a twist deformation along $1-6$, this would also be the case in $S$, a contradiction. Therefore, $S$ cannot exist and the proposition is proved. 
\end{proof}

\begin{prop} \label{prop-ring2} There is no $S\in\mathcal{T}$ such that $\mathrm{sys}(S)=\{1-2, 2-3, 3-4, 4-5, 3-6\}$. 
\end{prop}
\begin{proof}
The proof will be broken up into two pieces. 

 a) Suppose that such an $S$ exists and let $2x$ be the length of the five systoles. There are isometric $(0,3)$-subsurfaces $B$ and $B'$ of $S$ with boundary geodesics $1-2$, $4-5$, $3-6$. $B$ is composed of two right-angled hexagons $H$ and $H'$; we may assume that $3$ lies on $H$. Note that we may replace $x$ by $x+\varepsilon$ for $\varepsilon>0$ and by twist deformations along $1-2$ and along $4-5$, we may prolong the lengths of $2-3$ and $3-4$ to $2(x+\varepsilon)$. This can be done as long as there is not a further systole. Let $\sigma$be a further systole in  $\Sigma$. Then $\sigma$ cannot pass through $3$ since this would contradict Corollary \ref{cor-star} as there would be four systoles passing through $3$. Let $\sigma\in\{1-4, 2-5, 1-6, 5-6, 1-5\}$. Then $\mathrm{sys}(\Sigma)$ contains a disjoint pair and it follows from Proposition \ref{prop-pair} and Proposition \ref{prop-ring1} that $\mathrm{sys}(\Sigma)$ contains a triple. This is also the case if $\sigma\in \{2-4, 2-6, 4-6\}$. We therefore may assume that $S$ is in a neighborhood of a $\Sigma$ where $\mathrm{sys}(\Sigma)$ is a 1-skeleton of a triangular prism, see Proposition \ref{prop-triple}.

\begin{figure} 
\centering
\includegraphics[width=0.7\textwidth]{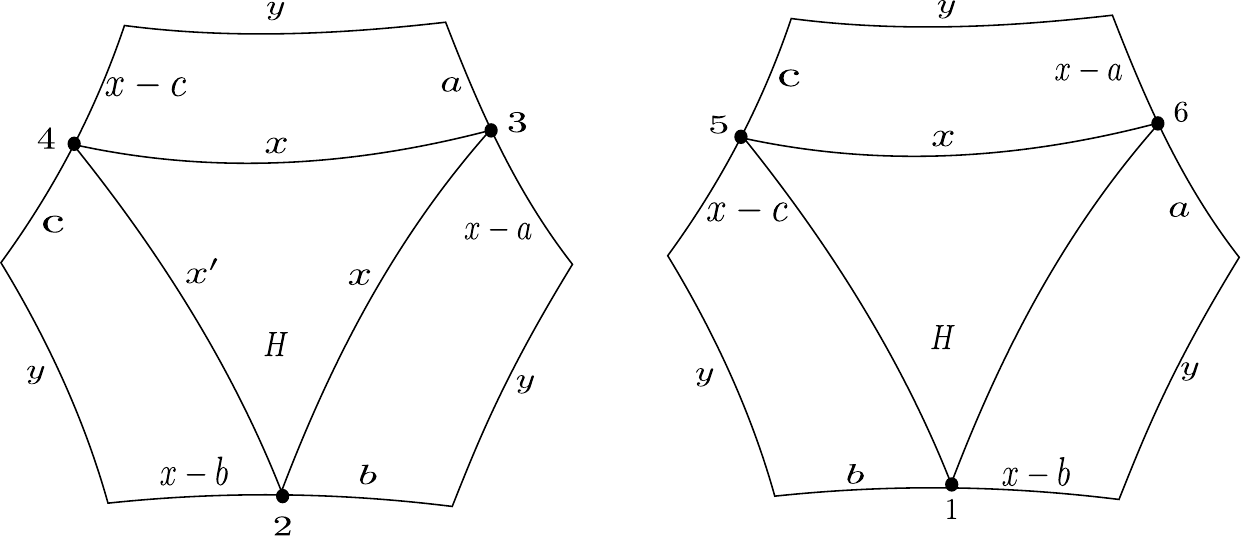}
\caption{Figure from part b) of the proof of Proposition \ref{prop-ring2}.}
\label{fig-ring3}
\end{figure}

\begin{figure} 
\centering
\includegraphics[width=0.7\textwidth]{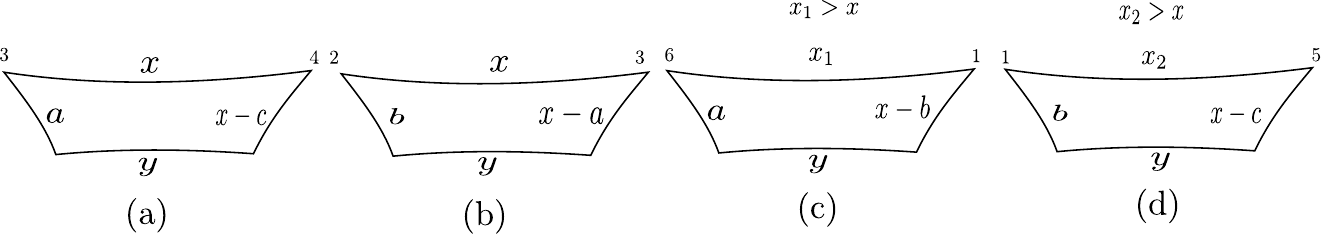}
\caption{The quadrilaterals from part b) of the proof of Proposition \ref{prop-ring2}.}
\label{fig-ring4}
\end{figure}

b) By Corollary \ref{cor-ring1} and Figure \ref{fig-ring2} we know how we may represent the systoles of $\Sigma$. Since $S$ is in a neighborhood of $\Sigma$ we may represent $S$ as in Figure \ref{fig-ring3}. We then obtain the quadrilaterals as in Figure \ref{fig-ring4}. Comparing Figure \ref{fig-ring4}(a) and Figure \ref{fig-ring4}(c), we see that $x-b>x-c$, hence $b<c$. Comparing Figure \ref{fig-ring4}(b) and Figure \ref{fig-ring4}(d), it follows that $x-a<x-c$ so that $c<a$ and thus $a>c>b$. On the other hand, comparing Figure \ref{fig-ring4}(a) and Figure \ref{fig-ring4}(b) and considering $a>b$, we have $x-c<x-a$ so that $a<c$. This is a contradiction, proving the proposition. 
\end{proof}

\begin{prop} \label{prop-ring3} There is no $S\in \mathcal{T}$ such that $\mathrm{sys}(S)$ is a long $Y$.
\end{prop}
\begin{proof}
The proof will be broken up into four parts. 

 a) Suppose that such an $S$ exists, let $\mathrm{sys}(S)=\{2-3, 3-4,4-5,4-6\}$ and let $2x$ be the length of the systoles. Let $z$ be the boundary geodesic of the $(1,1)$-subsurface of $S$ that contains $4-5$ and $4-6$. Let $1-2$ be the unique simple closed geodesic that does not intersect $3-4,4-5,4-6$, and let $1-3$ be the shortest geodesic such that $\{1-2,1-3, 2-3)\}$ is a triple.

b) It follows from the argument in part b) of Proposition \ref{prop-ring1} that the angle between $1-2$ and $2-3$ and the angle between $1-3$ and $2-3$ are bounded away from $\pi/2$.

\begin{figure} 
\centering
\includegraphics[width=0.3\textwidth]{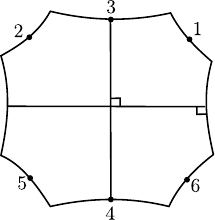}
\caption{The octagon $O$.}
\label{fig-ring6}
\end{figure}

c) Suppose the angles between $z$ and $3-4$ are $\pi/2$. Cutting $S$ along $5-6$ and $1-2$, we then obtain two isometric right-angled octagons $O$ and $O'$, where the octagon $O$ is shown in Figure \ref{fig-ring6} (see also part 2 of the proof of  Proposition \ref{prop-triple}). If the lengths of $1-3$ and $2-3$ in $O$  are equal, then 1 and 2 lie in the center of their respective sides in $O$. It then follows from the argument in part 1) of the proof of Proposition \ref{prop-pair} that the simple closed geodesics $1-6$ and $2-5$ are shorter than $3-4$. Since in $S$, $2-3$ is strictly shorter than $1-3$, it follows that in $S$, $1-6$ is strictly shorter than $3-4$. This shows that it is not possible that $z$ and $3-4$ intersect perpendicularly in $S$.

d) In $S$ we may execute a combined twist twist deformation along $z$ and along $2-3$ that keeps the length of $3-4$ unchanged and hence the lengths of all elements of $\mathrm{sys}(S)$ unchanged. Moreover, by b) and c), this combined twist deformation can be chosen such that the length of $1-2$ increases while the length of $1-3$ decreases. Again by b) and c), this can be done until there is a further systole $\sigma$ in $\Sigma$. Note that $\sigma$ must intersect $z$ and/or $2-3$. Note also that $\sigma$ cannot pass through 4 since this would contradict Corollary \ref{cor-star} (four systoles would pass through 4). If $\sigma\in\{3-5,3-6\}$, then $\mathrm{sys}(\Sigma)$ contains a triple. If $\sigma\in\{1-2,1-3\}$, then $\mathrm{sys}(\Sigma)$ contains a disjoint pair and hence a triple by Proposition \ref{prop-pair} and Proposition \ref{prop-ring1}. If $\sigma\in\{2-5,2-6\}$, then $\mathrm{sys}(\Sigma)$ contains a 4-ring which does not fill and hence a triple by Proposition \ref{prop-pair} and Proposition \ref{prop-ring1}. If $\sigma\in\{1-5,1-6\}$, then $\mathrm{sys}(\Sigma)$ contains a further element by Proposition \ref{prop-ring2} and thus a triple since $\mathrm{sys}(\Sigma)$ is then among the cases treated above. This shows that $\mathrm{sys}(\Sigma)$ is the 1-skeleton of a triangular prism, see Proposition \ref{prop-triple}.

Recall that $5-6$ is not a systole of $\Sigma$ since $5-6$ is not changed by our combined twist deformation. It follows that $\mathrm{sys}(\Sigma)$ contains the triple $\{3-4, 4-5, 3-5\}$ or the triple $\{3-4, 4-6, 3-6\}$. This implies that $\mathrm{sys}(\Sigma)$ also contains the triple $\{1-2,2-6,1-6\}$ or the triple $\{1-2, 2-5, 1-5\}$. However, by our combined twist deformation, $1-2$ cannot be a systole of $\Sigma$. This shows that $\Sigma$ and hence $S$ cannot exist and proves the proposition. \end{proof}

\begin{cor} \label{cor-ring3} Let $S\in\mathcal{T}$ be such that $\mathrm{sys}(S)$ contains a long $Y$. Then $\mathrm{sys}(S)$ contains a triple.
\end{cor}
\begin{proof}This was shown in part d) of the proof of Proposition \ref{prop-ring3}. 
\end{proof}

\begin{thm} \label{thm37}
Let $S\in\mathcal{T}$ be such that $\mathrm{sys}(S)$ fills. Then either $S$ is the Bolza surface or the systolic graph is the 1-skeleton of a triangular prism or the systolic graph is a 5-ring or a 6-ring or  a 4-chain. In all cases, the systolic graph contains a 4-chain.
\end{thm}
\begin{proof} Assume that the systolic graph contains a 3-star. Then it contains a long $Y$ since the systoles fill. It follows from Corollary \ref{cor-ring3} that the systolic graph contains a triple. It must therefore be the 1-skeleton of a triangular prism or $S$ is the Bolza surface, see Proposition \ref{prop-triple}.

We thus may assume that there are at most two systoles passing through each Weierstrass point. Consequently, the systolic graph is a 4-ring (which fills) or a 5-ring or a 6-ring or a 4-chain or a 5-chain. In the latter case the systolic graph contains a disjoint pair. It then follows from Proposition \ref{prop-pair} and Proposition \ref{prop-ring1} that the systolic graph cannot be a 5-chain. By Proposition \ref{prop-4ring}, it is impossible that the systolic graph is a 4-ring (which fills). This proves the theorem. 
\end{proof}

\begin{thm} \label{thm-irm} Let $C$ be a 4-star, a 4-chain, a 4-ring which fills, or a long $Y$. Then:
		\begin{enumerate}
			\item{ $E(C)$ contains, up to isometry, a unique element where the length of the geodesics of $C$ is minimal.}
			
			\item{If $C$ is a 4-star or a 4-ring that fills, the minimal element of part (1) is the Bolza surface. If $C$ is a 4-chain, then the minimal element of part (1) is $M_{6}$.}
			
			\item{Let $C$ be a long $Y$ and let $M_{Y}$ be the corresponding minimal element of part (1). Then the elements of $C$ are not the systoles of $M_{Y}$.}
		\end{enumerate}
\end{thm} 
	\begin{proof}
		1) By Corollary 3.10 of \cite{MS}, in order to prove part  (1), we have only to show that $E(C)$ is non-empty since, by Proposition \ref{prop-minimal}, $C$ is a minimal set that fills. Note that if $S$ is the Bolza surface, then $\mathrm{sys}(S)$ contains a 4-star, a 4-chain, a 4-ring (which fills) as well as a long $Y$, so that part (1) holds.
		
		2) Part (2) follows from Corollary \ref{cor-triple},  Corollary \ref{cor-star} and from Proposition \ref{prop-65}. 
		
		3) Part (3) is a consequence of Corollary \ref{cor-ring3}, of Proposition \ref{prop-triple}, and of the fact that $M_{9}$ is a critical point of $syst$. 
	\end{proof}

\begin{thm}
\label{newtheorem}
The only minimal filling sets realised as sets of systoles are 4-chains.
\end{thm}
\begin{proof}
This is a consequence of Proposition \ref{prop-minimal}, Corollary \ref{cor-star} and Propositions \ref{prop-4ring} and \ref{prop-ring3}.
\end{proof}

\begin{thm}
\label{thm-crit}
 Up to isometry, the function {\em syst} on moduli space has exactly four different critical points given by $M_{5}$, $M_{6}$, $M_{9}$ and the Bolza surface.
\end{thm}

\begin{proof} By Theorem 44 in \cite{SchmutzMorse} $M_{5}$, $M_{6}$, $M_{9}$ and the Bolza surface are critical points of $syst$. It remains to show that there are no further critical points.

Suppose $S$ is a critical point of $syst$. It is known that $syst(S)$ fills; this is a corollary of a proposition first proven in \cite{Thurston}. Since the systoles fill, $syst(S)$ must be one of the five sets described in Theorem \ref{thm37}. Moreover, $S$ must be such that the systoles in $S$  have the shortest length of all surfaces in $E(syst(S))$, otherwise $S$ is not critical by the convexity properties of the length functions. It then follows from Theorem \ref{thm-irm} that $syst(S)$ cannot be a 4-chain. Concerning the four other cases of Theorem \ref{thm37}, it follows from Proposition \ref{prop-triple} and its proof as well as from Proposition \ref{prop-65} that $S$ must be one of $M_{5}$, $M_{6}$, $M_{9}$ and the Bolza surface.  
\end{proof}

\begin{rem} Theorem \ref{thm-crit} appears in \cite{SchmutzMorse} as Theorem 44 (pg 443f). However, the proof that there are no other critical points is only sketchy and needed to be made rigorous. In particular, the possible existence of sets $syst(S)$ which contain a long $Y$, but not a triple, was not considered in \cite{SchmutzMorse}.
\end{rem}

	\section{Spine Computation}
	\label{spinecomputation}
	This section explains how we computed the  quotient of the Thurston spine by the action of the mapping class group, and shows the resulting 3-dimensional simplicial complex in pictorial form. The code, and instructions on how to use it can be found in the GitHub repository: \url{https://github.com/a-n-n-i/computations}.
	\subsection{Description of the algorithm}
	A key ingredient in this computation is local finiteness. Suppose a critical point $p$ has set of systoles $C$. Recall that there is then a neighbourhood of $p$ on which the systoles are all contained in the set $C$. In most cases, it is enough to use first derivatives of lengths of curves in $C$ to determine the strata adjacent to $p$. In a few cases where differentials of lengths were linearly dependent, it was necessary to use higher order derivatives.
	
	Differentials of lengths were computed with respect to the following parameterisation: A filling subset of $C$ must contain at least four curves. Cutting the surface along a 4-chain yields a 12-gon. It follows from hyperbolic trigonometry that parameterising the surface via this polygon requires three angle parameters and six edge length parameters, \cite{Ni}. Given any six parameters, the remaining three can be determined numerically by the following equation:
	\begin{equation*}
	\label{polygoneq}
	\prod_{i=1}^{12} M_i = I \quad \text{in } \mathrm{PSL}(2,\mathbb{R})
	\end{equation*}
	where $M_i$ is a matrix that represents a translation through a distance equal to the length of edge $i$ in the upper half plane model of the hyperbolic plane, followed by a rotation through the angle at the vertex at the end of edge $i$.
	
	The structure of the systoles around a critical point was studied in Section 5 of \cite{MS}. For the local maximum given by the Bolza surface, the gradients of the lengths of the twelve systoles span the tangent space to $\mathcal{T}$ at the point of $\mathcal{T}$ corresponding to the Bolza surface. However, if $p$ is one of the other three critical points, with set of systoles $C$, $\{\nabla L(c)(p)\ |\ c\in C\}$ does not span $T_{p}\mathcal{T}$, and it was shown in Section 5 of \cite{MS} that the unstable manifold of $p$ has tangent space given by the orthogonal complement of $\{\nabla L(c)(p)\ |\ c\in C\}$, which does not depend on the choice of metric used to define gradients. Moreover, the orthogonal complement of $\{\nabla L(c)(p)\ |\ c\in C\}$ is tangent to $\mathcal{P}$ at $p$. The critical point $M_{6}$ has index 3, and in this example it is not hard to see that there are no strata of $\mathcal{P}_{g}$ below $M_{6}$. The other critical points all have strata of $\mathcal{P}$ below them; these strata are contained in unstable manifolds of critical points with smaller index. Numerically, it is easier to compute the strata below a critical point, as this can mostly be done using first order derivatives.
	
	If $p$ is one of the four critical points, we are essentially calculating an analogue of a Voronoi decomposition $\mathcal{V}(p)$ of the unit tangent space to $\mathcal{T}$ at $p$. Let $C$ be the set of systoles at the critical point. A unit vector $\mathbf{v}$ in $T_{p}\mathcal{T}$ is in a cell of $\mathcal{V}(p)$ labelled by $C'\subset C$ if the inner product of $\mathbf{v}$ with $-\nabla L(c_{i})$ for all $c\in C'$ is equal and is less than the inner product of $\mathbf{v}$ with $-\nabla L(c_{i})$ for all $c_{i}\in C\setminus\{c\}$. Except in cases where higher order derivatives are needed, the cell decomposition $\mathcal{V}(p)$ determines the stratification near $p$. What we found was that this information was enough to piece together the strata making up $\mathcal{P}$. Every stratum is adjacent to at least two critical points.

	\subsection{The simplicial complex}
By a theorem of Lojasiewicz, \cite{Lojasiewicz1964}, there exists a triangulation of $\mathcal{P}$ compatible with the stratification, i.e. the interior of every simplex is contained in a single stratum. Our results are stated in the form of such a triangulation.
	
	Let $M_{12}$ be a marked surface representing the Bolza surface and let $C=\{c_1, c_2, \ldots, c_{12} \}$ be the set of systoles of $M_{12}$, as shown in Figure \ref{Bolza_surface}. Denote by $M_{9}$ the marked surface with systoles $\{c_1, c_2, c_3, c_5, c_6, c_8, c_9, c_{11}, c_{12} \}$, by $M_6$ the marked surface with systoles  $\{c_1, c_2, c_7, c_8, c_9, c_{11}\}$,  and by $M_{5}$ the marked surface with systoles $\{c_1, c_2, c_5, c_7, c_{11} \}$.

From \cite{SchmutzMorse} , we know that for the orientation preserving automorphism groups of critical points, we have:
	\[
	\operatorname{Aut}(M_{12}) \cong \operatorname{GL}(2,3),\quad
	\operatorname{Aut}(M_6) \cong 2D_6,\quad
	\operatorname{Aut}(M_9) \cong D_6,\quad
	\operatorname{Aut}(M_5) \cong \mathbb{Z}_{10}.
	\]
	To explicitly compute the elements of these automorphism groups, we represented the automorphisms as permutations of the systoles, and generated the entire automorphism group by finding some generators. Since the orders of the automorphism groups is known, we were able to verify that the entire group was obtained in this way.
	\begin{figure}[!ht]
		\includegraphics[width=0.4\textwidth]{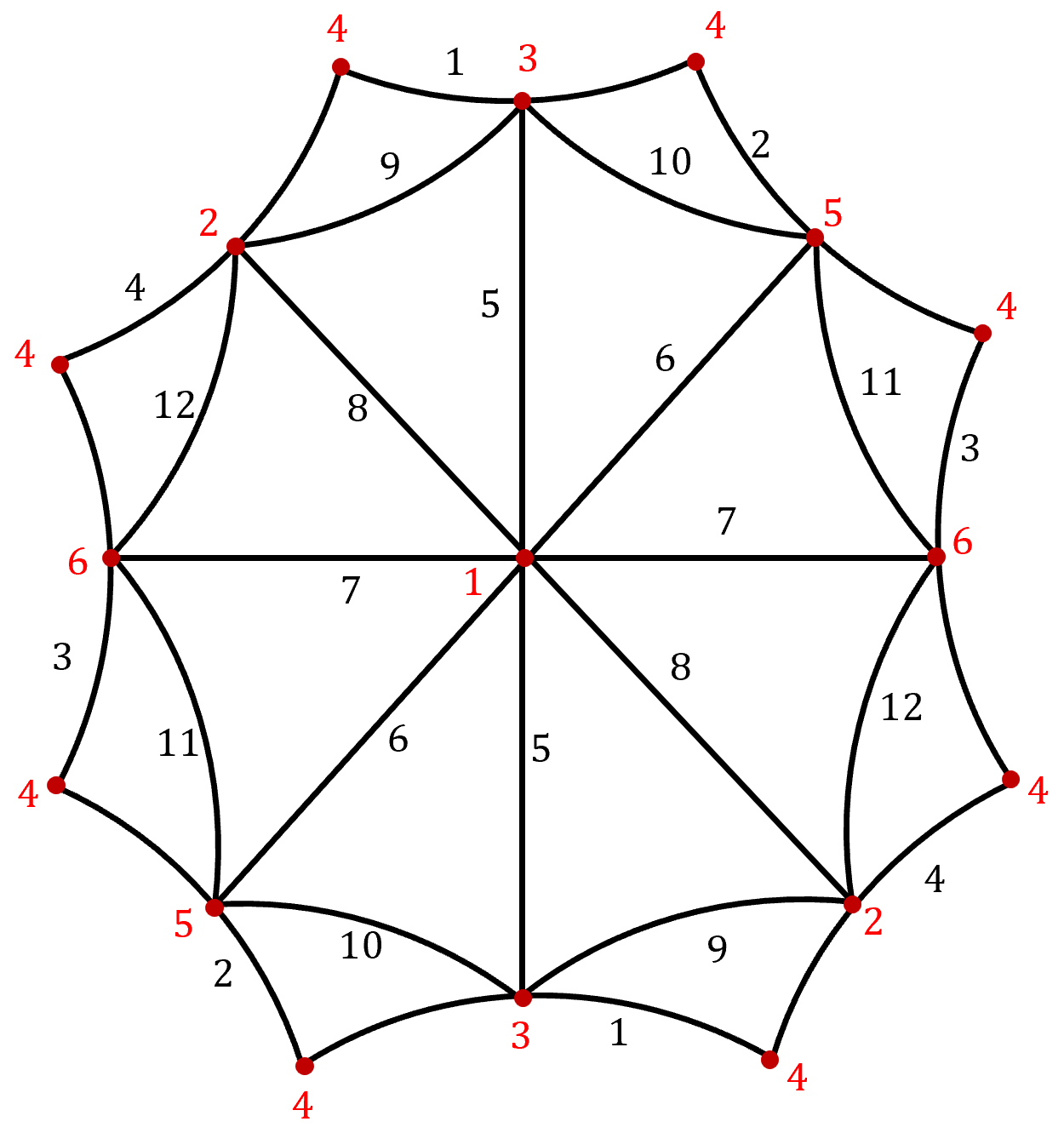}
		\caption{The Bolza surface $M_{12}$. The six red points are fixed points of the hyperelliptic involution, labelled 1 to 6. The curves marked with black numbers are the 12 systoles of $M_{12}$. }
		\label{Bolza_surface}
	\end{figure}
	
	In our triangulation of $\mathcal{P}$, $M_{12}$ is a vertex. Denote by $\mathrm{Aut}^{\pm}(M_{12})$ the group of automorphisms of the Bolza surface $M_{12}$ and by $\mathrm{Aut}(M_{12})$ the group of orientation-preserving automorphisms of $M_{12}$. The group $\mathrm{Aut}^{\pm}(M_{12})$ permutes the simplices with vertex on $M_{12}$. There are many elements of $M_{9}, M_{5}$ and $M_{6}$ under the action of $\mathrm{Aut}^{\pm}(M_{12})$. These have been labelled as shown below, where the sets of numbers correspond to the subsets of $\{c_1, c_2, c_3, c_5, c_6, c_{7},c_8, c_9, c_{10}, c_{11}, c_{12} \}$ that are systoles at the critical point.

	\begin{equation*}
		\begin{aligned}
			&M_6^1: \{1, 2, 7, 8, 9, 11\} 		&M_6^2: \{2, 3, 5, 8, 10, 12\}	\ \ 	&M_6^3:  \{3, 4, 5, 6, 9, 11\} 		&M_6^4:  \{1, 4, 6, 7, 10, 12\}\\
			&M_6^5: \{1, 2, 5, 8, 11, 12\} 		&M_6^6: \{2, 3, 5, 6, 9, 12\} 	\ \ 	&M_6^7: \{3, 4, 6, 7, 9, 10\} 		&M_6^8: \{1, 4, 7, 8, 10, 11\} \\
			&M_6^9: \{1, 4, 5, 6, 11, 12\} 		&M_6^{10}: \{2, 3, 7, 8, 9, 10\} 	\ \ 	&M_6^{11}: \{1, 2, 6, 7, 9, 12\} 		&M_6^{12}: \{3, 4, 5, 8, 10, 11\} \\
			&M_6^{13}: \{1, 3, 6, 8, 10, 12\} 	&M_6^{14}: \{2, 4, 5, 7, 9, 11\} 	\ \ 	&M_6^{15}: \{1, 3, 6, 8, 9, 11\} 		&M_6^{16}: \{2, 4, 5, 7, 10, 12\} 
		\end{aligned}
	\end{equation*}
	
	\begin{equation*}
		\begin{aligned}
			&M_9^1: \{1, 2, 3, 5, 6, 8, 9, 11, 12\} \ \ &M_9^2: \{2, 3, 4, 5, 6, 7, 9, 10, 12\} 	\ \ &M_9^3: \{1, 3, 4, 6, 7, 8, 9, 10, 11\}  \\	
			&M_9^4: \{1, 2, 4, 5, 7, 8, 10, 11, 12\}\ \ &M_9^5: \{1, 3, 4, 5, 6, 8, 10, 11, 12\} 	\ \ &M_9^6: \{1, 2, 3, 6, 7, 8, 9, 10, 12\}  \\
			&M_9^7: \{1, 2, 4, 5, 6, 7, 9, 11, 12\}	     	&M_9^8: \{ 2,3, 4, 5,  7,8, 9, 10, 12\}     \ \ & 
		\end{aligned}
	\end{equation*}
	
	\begin{equation*}
		\begin{aligned}
			&M_5^1: \{1, 2, 5, 7, 11\} \ \ 		&M_5^2: \{2, 3, 6, 8, 12\} \ \ 		&M_5^3: \{3, 4, 5, 7, 9\} \ \ 		&M_5^4: \{1, 4, 6, 8, 10\} \\
			&M_5^5: \{1, 4, 5, 7, 12\} \ \ 		&M_5^6: \{2, 3, 5, 7, 10\} \ \ 		&M_5^7: \{1, 3, 5, 6, 11\} \ \ 		&M_5^8: \{1,2,6,8,9\} \\
			&M_5^9: \{3, 4, 6, 8, 11\} \ \ 		&M_5^{10}: \{1, 3, 6, 7, 10\} \ \ 	&M_5^{11}: \{2, 4, 6, 7, 12\} \ \ 	&M_5^{12}: \{5, 8, 10, 11, 12\} \\
			&M_5^{13}: \{1, 3, 7, 8, 9\} \ \ 		&M_5^{14}: \{2, 4, 5, 8, 10\} \ \ 	&M_5^{15}: \{1, 2, 9, 11, 12\} \ \ 	&M_5^{16}: \{1, 3, 5, 8, 12\} \\
			&M_5^{17}: \{2, 4, 5, 6, 9\} \ \ 		&M_5^{18}: \{2,4,7,8,11\}  \ \ 		&M_5^{19}: \{1, 4, 10, 11, 12\}\ \ 	&M_5^{20}: \{5,6,9,11,12\}\\
			&M_5^{21}: \{7,8,9,10,11\} \ \ 		&M_5^{22}: \{ 6,7,9,10,12 \}\ \ 		&M_5^{23}: \{2,3,9,10,12\}\ \ 		&M_5^{24}: \{3,4,9,10,11 \}
		\end{aligned}
	\end{equation*}

	
	There are three $\text{Aut}^{\pm}(M_{12})$-orbits of strata adjacent to $M_{12}$ whose systoles are 4-chains. These 4-chains determine the systoles in the interiors of six distinct mapping class group orbits of 3-dimensional cells of $\mathcal{P}$. If we consider orbits of the action of the extended mapping class group, i.e. allowing orientation reversing automorphisms, there are only three orbits of 3-dimensional cells. A choice of 3 adjacent cells in these orbits have systoles given by
	$\{c_1,c_4,c_{10},c_{11}\}$, $\{c_1,c_4,c_{10},c_{12}\}$, $\{c_1,c_{10},c_{11},c_{12}\}$. 
	
	Figures \ref{1_4_10_12}, \ref{1_4_10_11} and \ref{1_10_11_12} show the three representatives of the orbits of 3-dimensional cells. In Figure \ref{1_10_11_12},  $M_6'$ is a critical point isomorphic to $M_6$, with systoles $c_1, c_8, c_{10}, c_{11}, c_{12}, c'$, where $c'$ is a simple closed geodesic passing through the Weierstrass points $2$ and $3$. A simplex is specified by the critical points on its boundary. The captions under the figures therefore show the systoles in the stratification on the simplex. These stratifications determine the gluing maps.
	
	\begin{figure}[!ht]
		\includegraphics[width=0.5\textwidth]{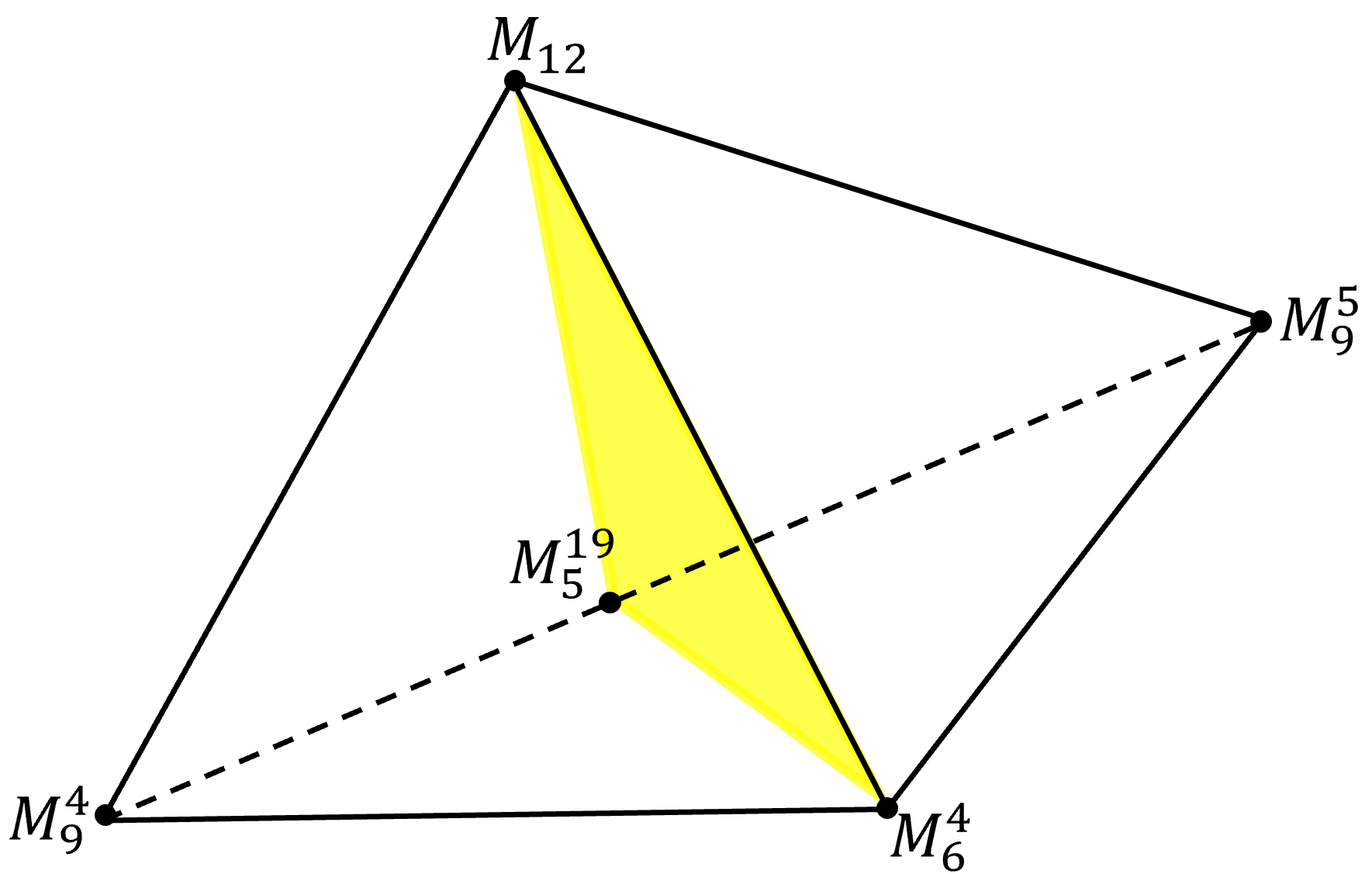}
		\caption{Simplex with vertex $M_{12}$ and interior in the stratum $\mathrm{str}(\{c_1,c_4, c_{10},c_{12}\})$.\\
			$M_{12}M_9^4M_5^{19}M_9^5, M_{9}^4M_5^{19}, M_9^5M_5^{19}, M_{12}M_5^{19}\in \mathrm{str}(\{c_1,c_4,c_{10},c_{11},c_{12}\})$(5-ring),\\
			$M_{12}M_9^4M_6^{4},M_9^4M_6^4 \in \mathrm{str}(\{c_1,c_4,c_7,c_{10},c_{12}\})$(5-chain),\\
			$M_{12}M_9^5M_6^4, M_9^5M_6^4 \in \mathrm{str}(\{c_1, c_4,c_6,c_{10},c_{12}\})$(5-chain),\\
			$M_{12}M_6^{4} \in \mathrm{str}(\{c_1,c_4,c_6,c_7,c_{10},c_{12}\})$(6-ring),\\
			$M_{12}M_9^4 \in \mathrm{str}(\{c_1,c_2,c_4,c_5,c_7,c_8,c_{10},c_{11},c_{12}\})=\mathrm{str}(\mathrm{sys}(M_{9}^{4}))$\\
			$M_{12}M_9^5 \in \mathrm{str}(\{c_1,c_3,c_4,c_5,c_6,c_8,c_{10},c_{11},c_{12}\})=\mathrm{str}(\mathrm{sys}(M_{9}^{5}))$.
		The yellow triangle is the fixed point set of the reflection $r$ defined below, which gives a symmetry of the simplex.}
		\label{1_4_10_12}
	\end{figure}

In Figure~\ref{1_4_10_12}, the yellow face divides this simplex into two simplices. The reflection $r$ defined below fixes the yellow face and is an automorphism of the simplex.
	
		 In Figure \ref{1_10_11_12}, the cell can be subdivided into two simplices. The face $M_{12}M_9^4M_6^{\prime}M_9^5$ is a 5-chain, and the face $M_{12}M_9^4M_5^{19}M_9^5$ is a 5-ring. These two faces have the same tangent cone at $M_{12}$. We guessed the sets of systoles by solving the equations to second order at $M_{12}$, and verified that the two faces have the sets of systoles claimed by solving first order equations at selected points along the edge from $M_{12}$ to $M_{9}^{4}$ and the edge from $M_{12}$ to $M_{9}^{5}$. The two faces $M_{9}^{4}M_{12}M_{9}^{5}M_{5}^{19}$ and $M_{9}^{4}M_{12}M_{9}^{5}M_{6}'$, as well as the  edges connecting different copies of $M_{9}$ were drawn as smooth. The justification for this is as follows: We know from Section 3 of \cite{MS} that any stratum labelled by a 4-chain is an open set contained in an embedded submanifold given by the locus of points at which the lengths of the curves making up the 4-chain are equal. The face $M_{9}^{4}M_{12}M_{9}^{5}M_{5}^{19}$ is contained in the locus of points on which the lengths of curves in a 5-ring are equal. This locus is determined by the three constraints that determine the locus labelled by the 4-chain, as well as one further constraint ensuring that the length of the fifth curve is equal to that of the other curves. This last constraint is linearly independent of the other constraints at $M_{5}^{19}$, and remains linearly independent at all other points we checked. If the stratum containing the interior of $M_{9}^{4}M_{12}M_{9}^{5}M_{5}^{19}$ is a smoothly embedded submanifold, we can also choose the edge $M_{9}^{4}M_{9}^{5}M_{5}^{19}$ contained in this stratum to be smooth. The face $M_{9}^{4}M_{12}M_{9}^{5}M_{6}'$ contained in the locus of points on which the lengths of curves in a 5-chain are equal. Although the four constraints defining this locus are not linearly independent at the critical point $M_{6}'$, this face also appears to be smooth. It is contained in the  smoothly embedded locus of points at which the lengths of curves in a 4-chain are equal, and numerical experimentation suggests it is locally a codimension 1 embedded submanifold of this locus, separating the points at which the four geodesics $\{c_{1}, c_{10}, c_{11}, c_{12}\}$ are shorter than $c_{8}$ from the points at which $\{c_{1}, c_{10}, c_{11}, c_{12}\}$ are longer than $c_{8}$. The face $M_{9}^{4}M_{12}M_{9}^{5}M_{6}'$ has $\mathrm{str}(\{c_1,c_{10}, c_{11},c_{12}\})$ on one side and $\mathrm{str}(\{c_8,c_{10}, c_{11},c_{12}\})$ on the other. The edge $M_{9}^{4}M_{6}'M_{9}^{5}$ is then chosen to be a smooth path contained in the stratum labelled by the 5-chain.

	\begin{figure}[!ht]
		\includegraphics[width=0.4\textwidth]{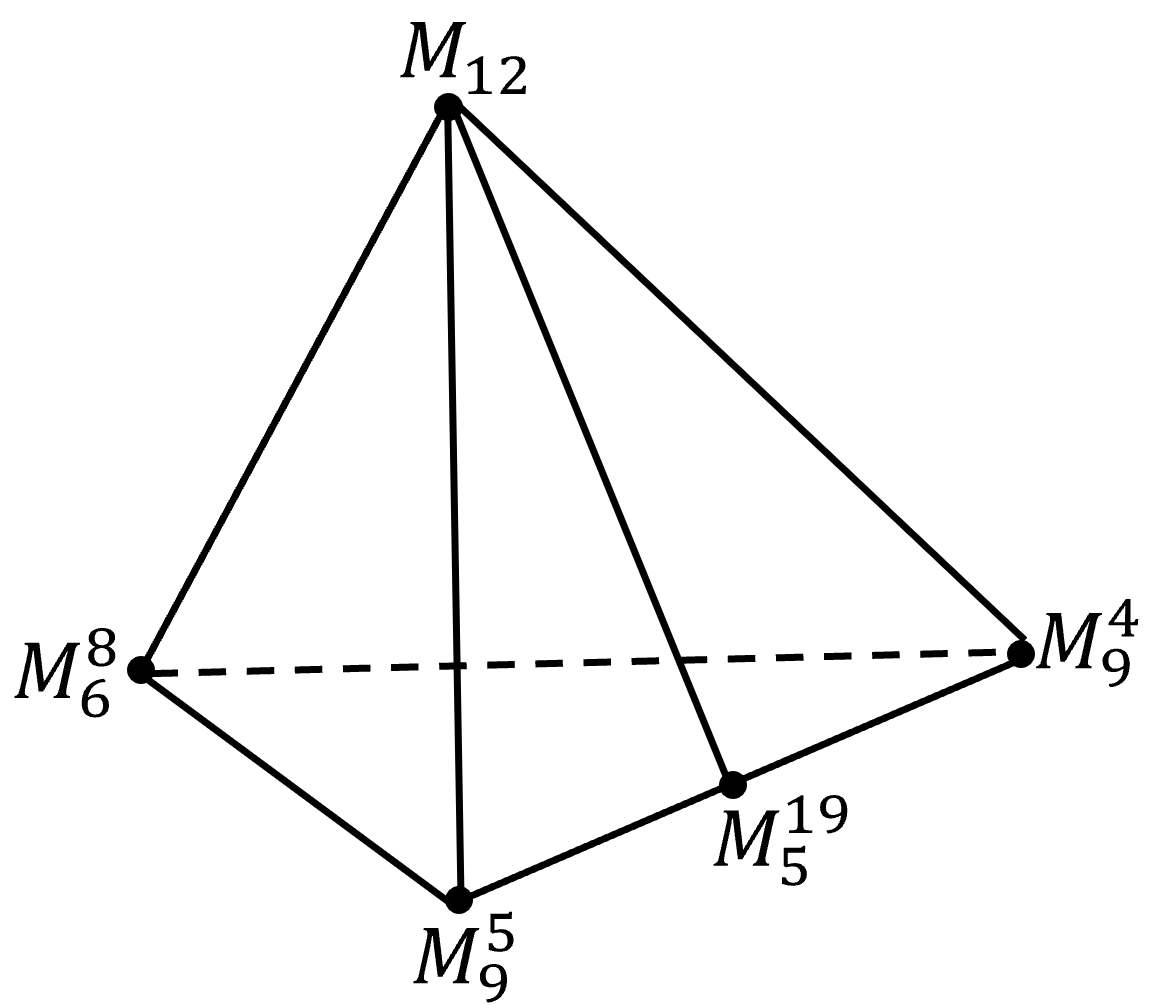}
		\caption{Simplex with vertex  $M_{12}$ and interior in the stratum $\mathrm{str}(\{c_1,c_4, c_{10},c_{11}\})$.\\
			$M_{12}M_9^4M_5^{19}M_9^5, M_{9}^4M_5^{19}, M_9^5M_5^{19}, M_{12}M_5^{19}\in \mathrm{str}(\{c_1,c_4,c_{10},c_{11},c_{12}\})$(5-ring),\\
			$M_{12}M_9^4M_6^{8},M_{12}M_6^8 \in \mathrm{str}(\{c_1,c_4,c_7,c_8,c_{10},c_{11}\})$(6-ring),\\
			$M_{12}M_9^5M_6^8 \in \mathrm{str}(\{c_1, c_4,c_8,c_{10},c_{11}\})$(5-chain),\\
			$M_{12}M_9^4 \in \mathrm{str}(\{c_1,c_2,c_4,c_5,c_7,c_8,c_{10},c_{11},c_{12}\})=\mathrm{str}(\mathrm{sys}(M_{9}^{4}))$,\\
			$M_{12}M_9^5 \in \mathrm{str}(\{c_1,c_3,c_4,c_5,c_6,c_8,c_{10},c_{11},c_{12}\})=\mathrm{str}(\mathrm{sys}(M_{9}^{5}))$.\\
		The straight line $M_{12}M_5^{19} $ is a fixed point set of the reflection $r$ defined below, which does not extend to an automorphism of the entire simplex.}
		\label{1_4_10_11}
	\end{figure}

	\begin{figure}[!ht]
		\includegraphics[width=0.4\textwidth]{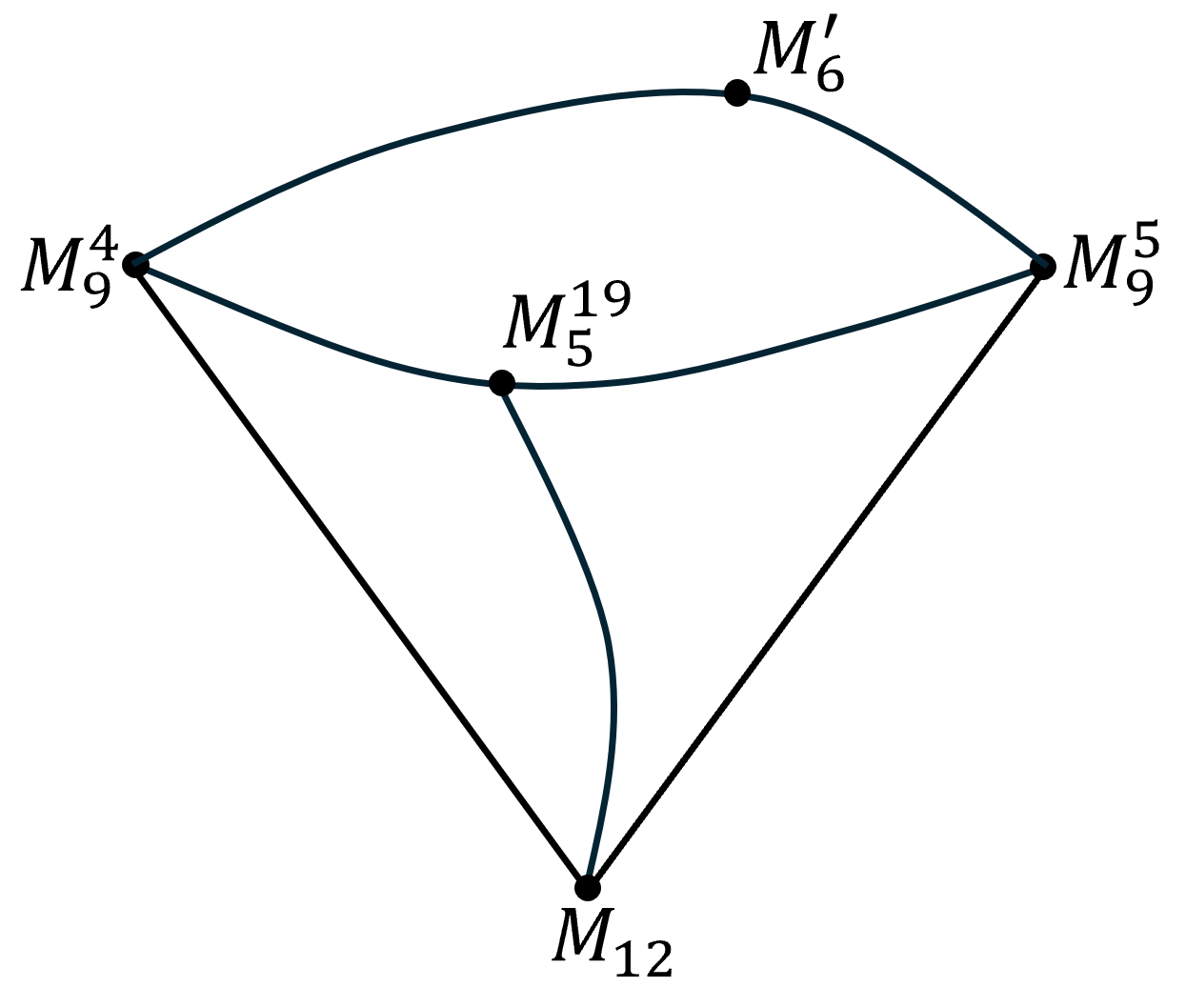}
		\caption{Cell with vertex $M_{12}$ contained in the stratum $\mathrm{str}(\{c_1,c_{10}, c_{11},c_{12}\})$.\\
			$M_{12}M_9^4M_5^{19}M_9^5, M_{9}^4M_5^{19}, M_9^5M_5^{19}, M_{12}M_5^{19}\in \mathrm{str}(\{c_1,c_4,c_{10},c_{11},c_{12}\})$(5-ring),\\
			$M_{12}M_9^4M_6^{\prime}M_9^5,M_9^4M_6^{\prime},M_9^5M_6^{\prime} \in \mathrm{str}(\{c_1,c_8,c_{10},c_{11},c_{12}\})$(5-chain),\\
			$M_{12}M_9^4 \in \mathrm{str}(\{c_1,c_2,c_4,c_5,c_7,c_8,c_{10},c_{11},c_{12}\})=\mathrm{str}(\mathrm{sys}(M_{9}^{4}))$,\\
			$M_{12}M_9^5 \in \mathrm{str}(\{c_1,c_3,c_4,c_5,c_6,c_8,c_{10},c_{11},c_{12}\})=\mathrm{str}(\mathrm{sys}(M_{9}^{4}))$.
		}
		\label{1_10_11_12}
	\end{figure}

	\begin{figure}[bh]
		\centering
		\includegraphics[width=0.95\linewidth]{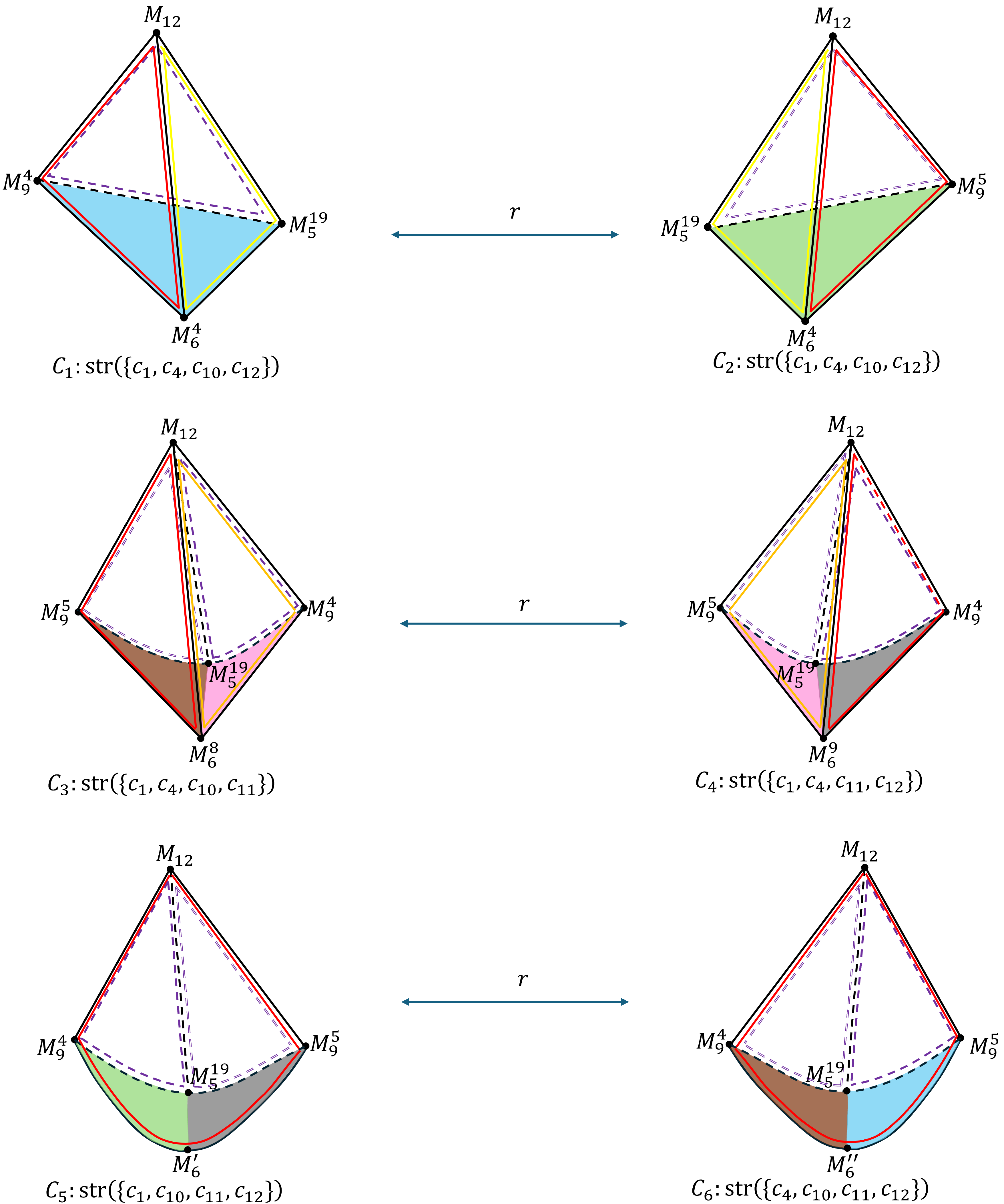}
		\caption{The cells of the cell decomposition of $\mathcal{P}/\Gamma_2$ showing the action of a reflection $r$ defined above. Gluing maps are indicated with colours.}
		\label{All_cell}
	\end{figure}

	\newpage
	Figure \ref{All_cell} shows the cells described above on a neighbourhood of the Bolza surface, where $r$ is an (orientation reversing) reflection of the Bolza surface in the extended mapping class group
	\begin{equation*}
		r(c_{1},c_{2},\ldots,c_{12})=(c_4,c_3,c_2,c_1,c_8,c_7,c_6,c_5,c_9,c_{12},c_{11},c_{10})
	\end{equation*}

	In Figure~\ref{All_cell}, faces painted the same color are glued together, and the gluing maps of faces sharing the same solid colour come from automorphisms of $M_5^{19}$. The faces outlined in red are contained in strata of $5$-chains, the faces outlined in purple are contained in strata of $5$-rings, the faces outlined in orange are contained in the stratum of a $6$-ring, and the two faces outlined in yellow are contained in strata of 4-chains which are glued together. The red and orange outlined faces are fixed point sets of groups shown in the Table \ref{isotropy_group} and are not glued to other simplices shown. Under the reflection, the yellow outlined simplex is fixed. The five purple outlined simplices with vertices $M_{12}, M_{9}^4$ and $M_{5}^{19}$ are glued together. The 5 faces outlined in double purple lines, with vertices $M_{12}$, $M_{9}^5$ and $M_{5}^{19}$ are also glued together.
	
	When calculating the strata around the critical point $M_{6}$, we found that in the example shown in Figure 12 of \cite{MS}, some of the strata are missing. The actual stratification around this critical point is more complicated than that given in \cite{MS}.
	
	\subsection{Isotropy and orbifold Euler characteristic}
	
	 For an orbifold, one can compute the orbifold Euler characteristic $\chi^{\mathrm{orb}}$ via a compatible triangulation: for each open cell $e$ in the triangulation with isotropy subgroup $G_e$,
	 $$
	 \chi^{\mathrm{orb}}(\mathcal{O}) = \sum_{e \subseteq \mathcal{O}} \frac{(-1)^{\dim e}}{|G_e|}.	 
	$$
	Define $\mathcal{M}_{g}$ to be the moduli space of orientable, closed, compact, connected surfaces of genus $g$. As shown in \cite{HarerZagier} and independently in \cite{Penner1988} the orbifold Euler characteristic of $\mathcal{M}_{g}$ is given by the formula:
$$
 \chi^{\mathrm{orb}}(\mathcal{M}_{g}) = \frac{(2g - 1)B_{2g}}{(2g)!} (2g - 3)!,
	$$
		So $ \chi^{\mathrm{orb}} (\mathcal{M}_{2})=-\frac{1}{240}$. The orbifold Euler characteristic for an orbifold is an orbifold-homotopy invariant. 
	
	Recall that we calculated the stratification around a critical point $p$ from the decomposition $\mathcal{V}(p)$ of the unit tangent space to $\mathcal{T}$ at $p$.

Note that for $\gamma$ an element of the mapping class group, we can represent $\gamma$ as a matrix acting on $T_p\mathcal{T}$. If $\gamma$ fixes a stratum, the corresponding matrix also fixes every vector corresponding to a point  in the Voronoi decomposition associated with this stratum. In this way we can compute the isotropy of the cells in Figure \ref{All_cell}. The structure of the fixed point sets of the mapping class group in Teichm\"uller space was discussed in detail in Section 3.1 of \cite{JustVCD}. Using the fact that the topological Morse function $syst$ is mapping class group-equivariant, it was shown that fixed point sets are either critical points, or higher dimensional connected embedded submanifolds transverse to noncritical level sets and each containing at least one critical point. Consequently, we can find all fixed point sets by studying the fixed point sets around the critical points.
	
	Select a representative from every mapping class group orbit of each cell. The Table \ref{isotropy_group} lists the cells of all dimensions together with their corresponding isotropy groups, from which we can compute that $ \chi^{\mathrm{orb}} (\mathcal{P}/\Gamma_2)=-\frac{1}{240}$. 
	\begin{table}[bh]
		\begin{tabular}{
				l |
				l |
				>{\raggedright\arraybackslash}p{10cm}   
			}
			
			\hline
			& Isotropy group & Cells  \\ \hline
			\multirow{4}{*}{0-cells} & $GL(2,3)$ &$M_{12}$  \\ \cline{2-3} 
			& $2D_6$ &$M_6$ \\ \cline{2-3} 
			& $D_6$ &$M_9$ \\ \cline{2-3} 
				& $\mathbb{Z}_{10}$ &$M_5$ \\ \hline
			\multirow{4}{*}{1-cells} & $D_6$ & $M_{12}M_9^4$, $M_{12}M_9^5$, $M_{12}M_6^4$  \\ \cline{2-3} 
			& $D_4$ & $M_{12}M_6^8$, $M_{12}M_6^9$  \\ \cline{2-3} 
			& $\mathbb{Z}_2 \times \mathbb{Z}_2$ &  $M_9^4M_6^4$, $M_9^4M_6^9$, $M_9^4M_6^{\prime}$, $M_9^4M_6^{\prime \prime}$,  $M_9^4M_6^8$  \\ \cline{2-3} 
			& $\mathbb{Z}_2$ &  $M_{12}M_5^{19}$, $M_9^4M_5^{19}$,  $M_6^4M_5^{19}$ \\ \hline
			\multirow{4}{*}{2-cells} &$\mathbb{Z}_2 \times \mathbb{Z}_2$ &$M_{12}M_9^4M_6^4$, $M_{12}M_9^5M_6^4$, $M_{12}M_6^{8}M_9^5$, $M_{12}M_6^8M_9^4$, $M_{12}M_6^9M_9^5$, $M_{12}M_6^9M_9^4$, $M_{12}M_9^{4}M_6^\prime M_9^5$,  $M_{12}M_9^{4}M_6^{\prime \prime} M_9^5$ \\ \cline{2-3} 
			& $\mathbb{Z}_2$ &$M_{12}M_9^4M_5^{19}$, $M_{12}M_9^5M_5^{19}$, $M_{12}M_5^{19}M_6^4$,  $M_{9}^4M_6^4M_5^{19}$,  $M_{9}^5M_6^4M_5^{19}$, $M_{9}^4M_6^8M_5^{19}$,  $M_{9}^5M_6^8M_5^{19}$,  $M_{9}^4M_6^9M_5^{19}$ \\ \hline
			\multirow{1}{*}{3-cells} &$\mathbb{Z}_2$ & $C_1$, $C_2$, $C_3$, $C_4$, $C_5$, $C_6$  \\ \hline
		\end{tabular}
		\caption{The cells of all dimensions together with their corresponding isotropy groups.}
		\label{isotropy_group}
	\end{table}

	\section{Cell decompositions of moduli space}
	\label{jiexing}
	There are two different proven ways of constructing cell decompositions of $\mathcal{M}_{2}$, both of which are not well-known. Subsection \ref{paulscelldecomp} surveys a cell decomposition from \cite{SchmutzVoronoi}, which is the simplest possible. The second cell decomposition uses the theoretical framework of \cite{RivinII}. In \cite{Armando} it was pointed out that this framework could be used to describe all possible Voronoi/Delaunay decompositions of the orbifold obtained as the quotient of a genus 2 hyperbolic surface by the hyperelliptic involution. In \cite{Armando} the top-dimensional cells and the identifications between top-dimensional cells were determined. We computed the entire cell decomposition, including all gluing maps as well as cells in the bordification. The angle coordinates used to describe these cells are discussed in Subsection \ref{anglecoordssec}. Subsection \ref{jiexingcomp} outlines the details of the computation of this cell decomposition. The complete computation, together with the SageMath code required to reconstruct the cell decomposition, is available in the GitHub repository \url{https://github.com/Jason-500/ModuliCells}. Finally the four critical points of the systole function are described in terms of angle coordinates in Subsection \ref{criticalanglesec}.

	\subsection{A cell decomposition from sets of minima}
	\label{paulscelldecomp}
	The purpose of this subsection is to describe the cell decomposition of $\mathcal{M}$ from \cite{SchmutzCelldecomp1}. This cell decomposition is defined in terms of sets of minima defined in Subsection \ref{subdefns}.
	
	It was shown in \cite{SchmutzMorse} that $\mathrm{Min}(C)$ is a continuously differentiable cell in $\mathcal{T}$ when certain ``regularity conditions'' are fulfilled. As discussed below, these regularity conditions are fulfilled on the set of minima $\mathrm{Min}(C_{7})$, where $C_7=\{c_1, c_2, c_3, c_4, c_{10}, c_{11}, c_{12} \}$ is a seventuple, contained in the set of systoles of $M_{12}$. The systolic graph of $C_{7}$ is shown in Figure \ref{graphc7fig}. The set $C_{7}$ clearly fills, because $C_{7}$ contains a 4-chain.
	
	\begin{figure}[!thpb]
		\centering
		\includegraphics[width=0.4\textwidth]{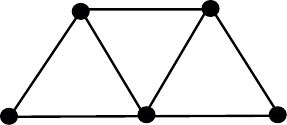}
		\caption{The systolic graph of $C_{7}$.}
		\label{graphc7fig}
	\end{figure}
	
	It was shown in Lemma 4 of \cite{SchmutzVoronoi} that the lengths of curves in $C_{7}$ parameterise $\mathcal{T}$. That is to say, knowing the lengths of all the curves in $C_{7}$ uniquely determines the point in $\mathcal{T}$. It follows from Theorem 1 of \cite{SchmutzVoronoi} that $\mathrm{Min}(C_{7})$ is 6-dimensional. The fact that the lengths of curves in $C_{7}$ parameterise also ensures that the regularity conditions of Theorem 15 of \cite{SchmutzMorse} are satisfied. Consequently, $\mathrm{Min}(C_{7})$ is a 6-dimensional cell, and the boundary of $\mathrm{Min}(C_{7})$ is a disjoint union of sets of minima of filling subsets of $C_{7}$. One can show that each of the subsets of $C_{7}$ of cardinality 6 fill, and determine the seven 5-dimensional boundary cells of $\mathrm{Min}(C_{7})$. These 5-dimensional boundary cells meet along boundary cells corresponding to smaller filling subsets of $C_{7}$.
	
	Theorem 5 of \cite{SchmutzVoronoi} states that the mapping class group orbits of $\mathrm{Min}(C_{7})$ and its boundary cells give a cell decomposition of $\mathcal{T}$. This is proven by using the observation that for each element $c_{i}$ in $C_{7}$, $C_{7}\setminus\{c_{i}\}$ is contained in a unique seventuple $C_{7}^{i}$, with $C_{7}^{i}\neq C_{7}$. $\mathrm{Min}(C_{7}^{i})$ is the interior of the 6-dimensional cell adjacent to $\mathrm{Min}(C_{7})$, sharing the common boundary cell $\mathrm{Min}(C_{7}\setminus\{c_{i}\})$.
	
	\subsection{Angle Coordinates}
	\label{anglecoordssec}
	A hyperbolic metric on $S$ projects onto a singular hyperbolic metric on $S/\tau$, and this metric determines a Voronoi decomposition of $S/\tau$ based around the Weierstrass points. This Voronoi decomposition is dual to a Delaunay decomposition. A Delaunay decomposition determines a set of circumcircles of the cells in the decomposition, with the property that no vertex of the Delaunay decomposition lies in the interior of any circumcircle. The edges of the Delaunay decomposition have their endpoints on pairs of intersections of circles. For each edge of the Delaunay decomposition, it is therefore possible to define an angle, as shown in Figure \ref{fig:angleparameter}. The angle is the same at each endpoint of an edge $e$ of the Delaunay decomposition, so the edge $e$ can be unambiguously labelled by this angle $\theta_{e}$. The graphs on $S/\tau$ that arise as Delaunay decompositions in this way are completely characterised by the constraints derived in \cite{RivinI}, summarised in \cite{Springborn}. 
	
	\begin{figure}[htbp]
    \centering
    \begin{tikzpicture}
      \coordinate (O1) at (0,0);
      \coordinate (O2) at (2,0);
      \draw (O1) circle (1.5);
      \draw (O2) circle (1.5);

      \coordinate (A) at (1,  1.118);
      \coordinate (B) at (1, -1.118);

      \draw[thick] (A) -- (B) node[pos=0.4, right] {$e$};

      \coordinate (P1) at ($(O1) + (180:1.5)$);
      \coordinate (P2) at ($(O2) + (20:1.5)$);
      \coordinate (P3) at ($(P1) + (100:1.2)$);
      \coordinate (P4) at ($(P1) + (260:1.2)$);
      \coordinate (P5) at ($(P2) + (80:1)$);
      \coordinate (P6) at ($(P2) + (275:1.8)$);
      \coordinate (P7) at ($(B) + (190:2)$);
      \coordinate (P8) at ($(B) + (350:2)$);
      \draw[thick,dashed] (A) -- (P1) (B) -- (P1);
      \draw[thick,dashed]  (A) -- (P2) (B) -- (P2);
      \draw[thick,dashed] (P1) -- (P3) (P1) -- (P4);
      \draw[thick,dashed] (P2) -- (P5) (P2) -- (P6);
      \draw[thick,dashed] (B) -- (P7) (B) -- (P8);
      \fill (P3) node[above, xshift=0pt, yshift=0pt] {$D$};

      \coordinate (TA1) at ($(A)!0.8!-90:(O1)$);
      \coordinate (TA2) at ($(A)!0.8!90:(O2)$);
      \coordinate (TB1) at ($(B)!-0.8!-90:(O1)$);      
      \coordinate (TB2) at ($(B)!-0.8!90:(O2)$);

      \draw (A) -- (TA1);
      \draw (A) -- (TA2);
      \draw (B) -- (TB1);
      \draw (B) -- (TB2);

      \pic[draw, -, angle radius=5mm, pic text={$\theta_e$}] {angle = TA2--A--TA1};
      \pic[draw, -, angle radius=5mm, pic text={$\theta_e$}] {angle = TB1--B--TB2};

    \end{tikzpicture}
    \caption{The angle parameter $\theta_e$ corresponding to the edge $e$ in a Delaunay triangulation $D$. Figure taken from \cite{JChen}.}
    \label{fig:angleparameter}
	\end{figure}
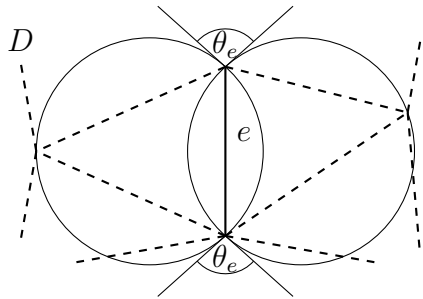
	
	\begin{thm}[Simplified version of Theorem 1.8 of \cite{Springborn}]
		Let $\Sigma$ be a cell decomposition of $S/\tau$ with vertices at the cone singularities. Suppose also that each edge of $\Sigma$ is labelled by an angle $\theta\in (0,\pi)$. We require that the cone singularities at the vertices of $\Sigma$ are all equal to $\pi$.
		A Euclidean circle pattern corresponding to this data exists if and only if for every nonempty subset
		$F'\subset F$ of the set $F$ of faces, and $E'$ a subset of the set of all edges $E$ with the property that an edge in $E'$ is incident with a face  $f \in  F'$, then 
		\begin{equation}
		\label{constrainteqns}
			\pi\mid F'\mid \leq \sum_{e\in E'}2\theta_{e}
		\end{equation}
		where $\mid F'\mid$ is the cardinality of $F'$, and equality holds if and only if $F'=F$.\\
		If it exists, the circle pattern is unique up to similarity. \\
		A hyperbolic circle pattern on $S/\tau$ corresponding to this data exists if and only if in the above condition, strict inequality also holds when $F'=F$. If it exists, the circle pattern is unique up to hyperbolic isometry.
	\end{thm}
	
	If a Voronoi graph dual to a graph satisfying the constraints in Equation (\ref{constrainteqns}) has the property that all sufficiently small deformations of the hyperbolic metric leave the graph unchanged, the Voronoi graph must have valence 3. The top-dimensional (in our case, six dimensional) cells in a cell decomposition of $\mathcal{T}$ therefore correspond to Delaunay triangulations of $S/\tau$ satisfying the constraints in Theorem \ref{constrainteqns} whose pre-images in $S$ give Delaunay triangulations of $S$. The computer programm Plantri can find all such graphs. An algorithm is also given in \cite{JChen}. Euler characteristic arguments give that the Voronoi decomposition of the sphere with singular hyperbolic structure must have 8 vertices, 12 edges and 6 faces.
	
	The graphs representing lower dimensional cells in the interior of the cell decomposition are obtained by deleting edges of the Delaunay graphs representing top dimensional cells, or equivalently, collapsing edges of the Voronoi decomposition. In the generic case, the dimension of the cell represented by a graph is reduced by one each time an edge is deleted, and the facets of a cell correspond to graphs, each of which is obtained by deleting one edge of the graph corresponding to a cell. The exception to this is given by bipartite graphs. If a graph is bipartite, one of the constraints that the sum of the angles at each of the vertices is equal to $\pi$ becomes dependent. When a bipartite graph is obtained, it is therefore possible to reduce the dimension of a cell corresponding to a graph by deleting two edges of the graph.
	
	The graphs obtained in this way give cells representing orbits of the action of the mapping class group on $\mathcal{T}$. To obtain a cell decomposition of $\mathcal{M}_{2}$, we take one cell for each orientation preserving planar isomorphism class of graphs. We also need to understand the action of the mapping class group on these cells. It can happen that a cell corresponds to a graph with nontrivial orientation preserving planar automorphisms. These automorphisms correspond to elements of the mapping class group that map the cell to itself. To obtain a cell in the cell decomposition of moduli space, it is necessary to take the quotient of the cell  by this action.
	
\begin{figure}[!thpb]
		\centering
		\includegraphics[width=0.5\textwidth]{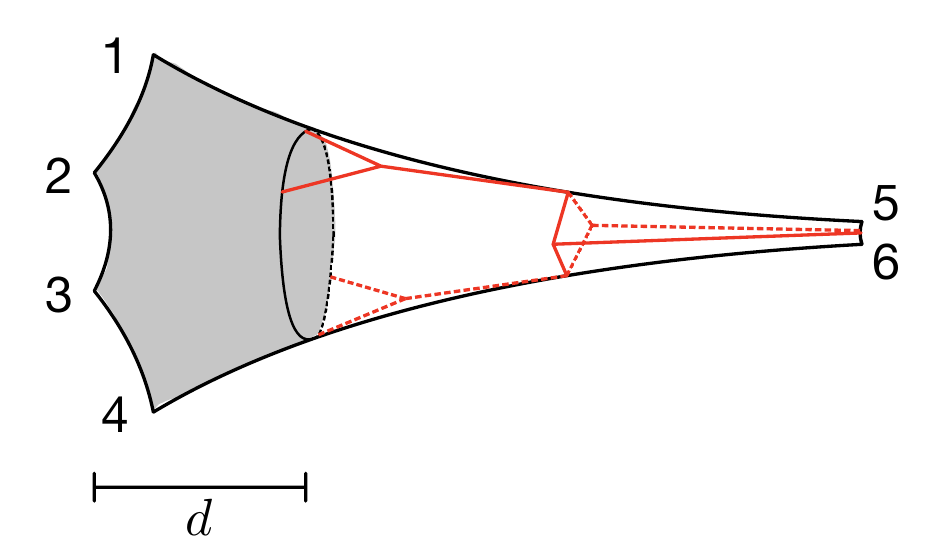}
		\caption{The figure shows the Voronoi graph (red) on the 2-sphere with cone points corresponding to a nonseparating geodesic on the surface connecting the Weierstrass points 5 and 6 that is made short. }
		\label{nonseparating1}
	\end{figure}

\begin{figure}[!thpb]
		\centering
		\includegraphics[width=0.5\textwidth]{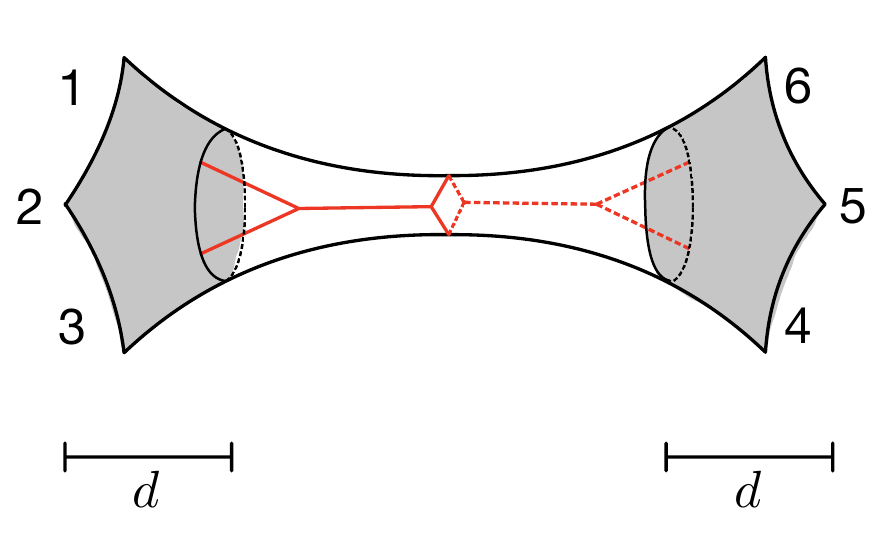}
		\caption{The figure shows the Voronoi graph (red) on the 2-sphere with cone points corresponding to a separating geodesic on the surface. This separating curve has the Weierstrass points 1,2,3 on one side and 4,5,6 on the other. }
		\label{separating1}
	\end{figure}

	Theorem \ref{constrainteqns} does not apply to graphs in the bordification representing noded surfaces. The idea behind how to obtain a cell decomposition of the bordificaton of $\mathcal{T}$ will now be explained; the full details with proofs are given in \cite{JChen}. Choose a top-dimensional cell corresponding to a valence 3 embedded planar graph. This cell can be described as a six dimensional polytope in $\mathbb{R}^{12}$ (Note that there are 12 angle parameters; one for each edge.) Some of the faces of this polytope do not represent cells interior to $\mathcal{T}$. Let $p$ be a point in the thin part of a cell. The corresponding surfaces can be decomposed into a union of long, thin collars of at most 3 short geodesics, as well as a ``thick'' part with diameter bounded from above by $d$. This is represented in Figures \ref{nonseparating1} and \ref{nonseparating}.
	
	\begin{figure}[!thpb]
		\centering
		\includegraphics[width=0.8\textwidth]{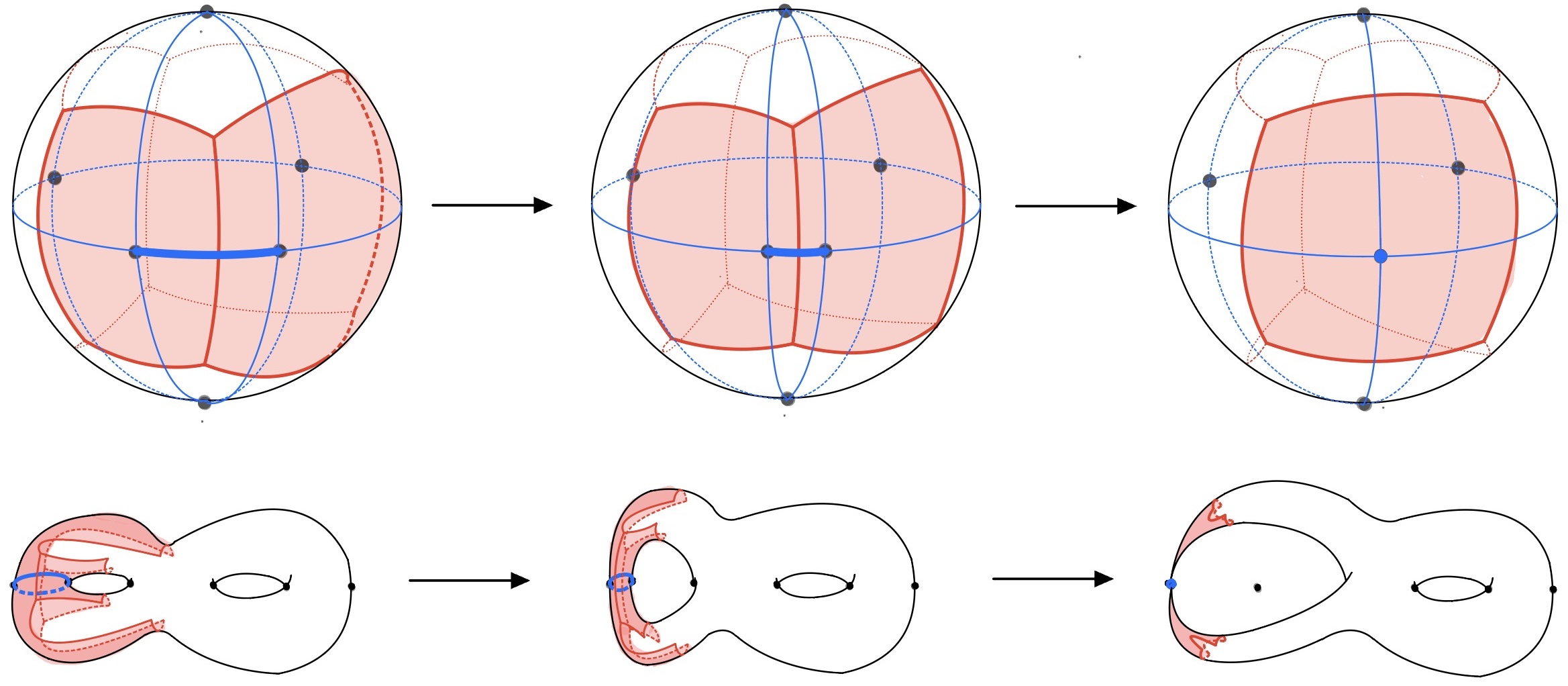}
		\caption{The figure shows an edge of the Delaunay graph (Blue) on the 2-sphere with cone points corresponding to a nonseparating geodesic on the surface that is made short. }
		\label{nonseparating}
	\end{figure}

\begin{figure}[!thpb]
		\centering
		\includegraphics[width=0.8\textwidth]{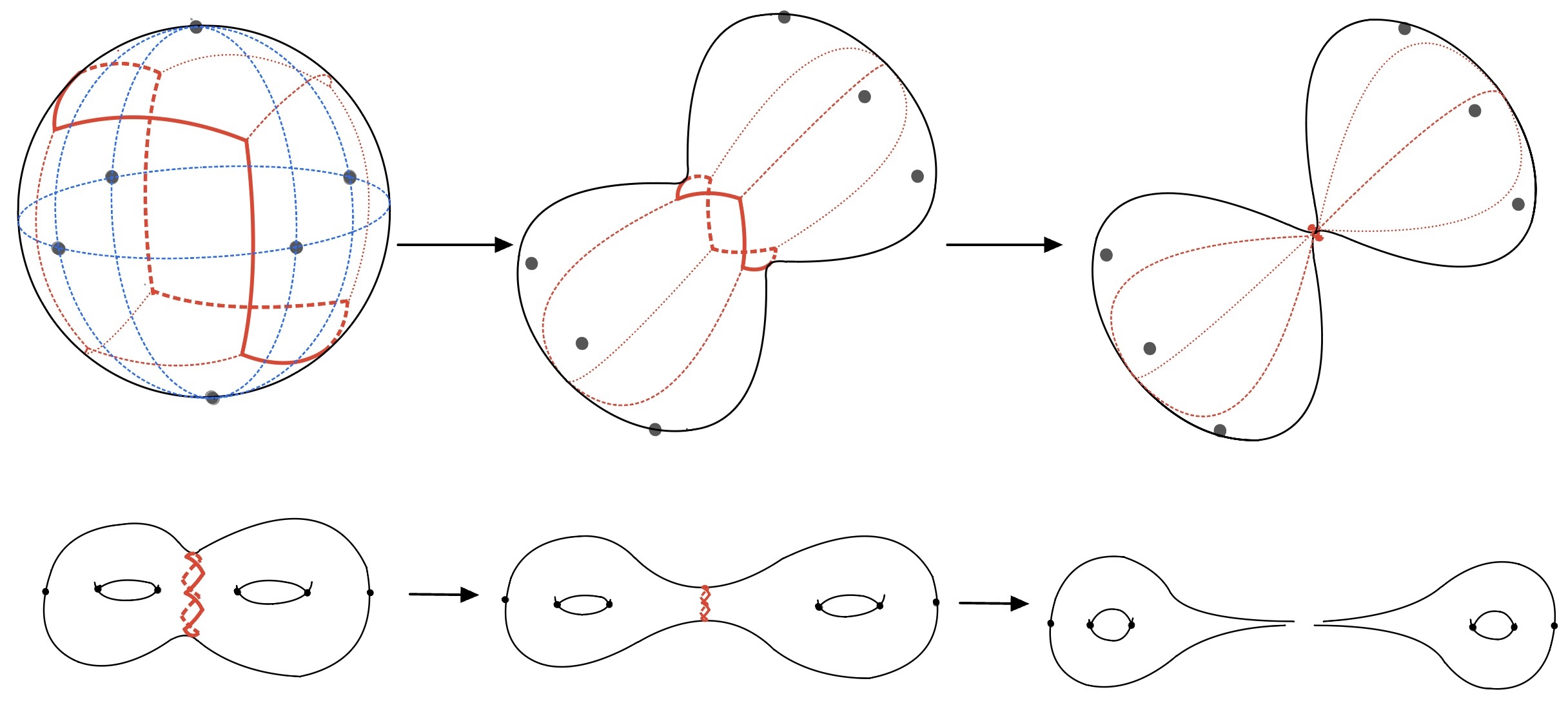}
		\caption{The figure shows the Voronoi graph (red) on the 2-sphere with cone points corresponding to a separating geodesic on the surface that is made short. }
		\label{separating}
	\end{figure}
	
	When only one geodesic $c$ is made arbitrarily short (this is called ``pinching'' the geodesic $c$) in the separating case, $c$ does not pass though any Weierstrass points, and the length $l$ of the collar of $c$ is large in comparison with $d$. Consequently there is a cycle in the Voronoi graph homotopic to $c$, that is collapsed to a point as $c$ is pinched. In the nonseparating case, the curve must pass through two distinct Weierstrass points. Since the length of the collar $l$ is much larger than $d$, there is a cycle in the Voronoi graph homotopic to $c$ and approximately half way along the collar of $c$, as shown in Figure \ref{nonseparating1}. This cycle is collapsed to a point as $c$ is pinched. The nonseparating case is illustrated in Figures \ref{nonseparating1} and \ref{nonseparating} and the separating case is illustrated in the Figures \ref{separating1} and \ref{separating}.
	
	We obtain similar results when more than one geodesic is pinched. Note that by the collar lemma, it is only possible to pinch disjoint geodesics, so the different cases to consider are: pinching two or three nonseparating geodesics, and pinching a separating geodesic and one or two nonseparating geodesics. In these cases, it is necessary to consider the various rates at which the geodesics are pinched. In all of these cases but two, the Voronoi graph contains a cycle homotopic to each of the geodesics being pinched, each of which is collapsed to a point as the geodesics are pinched. The two remaining cases involve pinching three disjoint nonseparating curves. When these curves are not pinched at the same rate, the Voronoi graph contains a cycle homotopic to each of the geodesics $c_{1}$ and $c_{2}$ that are pinched the fastest, where these two cycles are collapsed to a point as $c_{1}$ and $c_{2}$ are pinched. The final case is very symmetric; three nonseparating geodesics, each being pinched at the same rate. In this case, the Delaunay graph is a hexagon, with alternating edges shrinking to a point as the geodesics are pinched.
	
	After going through all these different cases, we can identify the degenerate graphs that we use to label the cells on the boundary of the top-dimensional polytopes representing noded surfaces in the bordification of the cell. As these graphs are obtained by taking limits of the graphs describing the interior points of the cell, the action of the mapping class group extends to the bordification.

	\subsection{Computational details}
	\label{jiexingcomp}
	The computation of the cellular decomposition described above was carried out in SageMath as part of the Masters thesis project of J. Chen \cite{JChen}.  The computation proceeds in two main steps.

	First, we determine the top-dimensional cells. All lower-dimensional cells can then be obtained from these top-dimensional cells.

	Recall that a condition for a Delaunay graph to determine a top-dimensional cell is that every face of the decomposition is triangular. This determines the relevant combinatorial parameters: such a graph must have $12$ edges, $8$ faces, and $6$ vertices.

	Since the size of this data set is modest, we can exhaustively enumerate all Delaunay-Voronoi graph pairs satisfying these combinatorial conditions. This yields $20$ candidates in total, as shown in Figure \ref{20pairs}. After imposing the constraints from Theorem \ref{constrainteqns}, only the $10$ graphs indicated in Figure \ref{20pairs} satisfy the required conditions. 
	\begin{figure}
		\centering
		\includegraphics[width=0.8\textwidth]{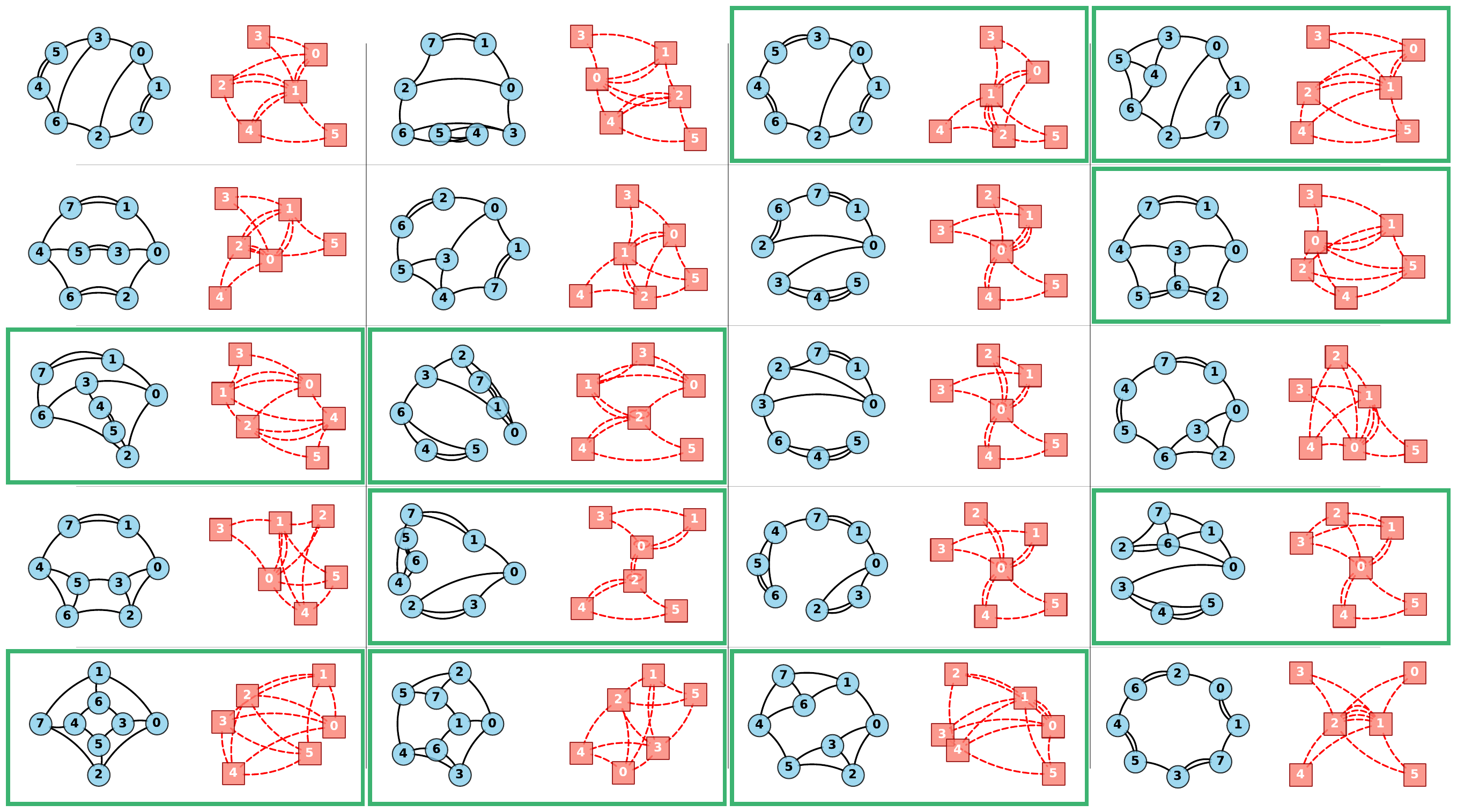}
		\caption{The 20 Delaunay (Pink) and Voronoi (Blue) graph pairs described above. The 10 satisfying the constraints from Theorem \ref{constrainteqns} are outlined in green. This graph was also known to \cite{Armando}.}
		\label{20pairs}
	\end{figure}
	We then compute the lower-dimensional cells. By Theorem \ref{constrainteqns}, the feasible region determined by the conditions associated with a given Delaunay-Voronoi graph pair is a polyhedron, and this polyhedron completely determines the geometry of the corresponding cell. In particular, the faces of the polyhedron are in strict correspondence with the faces of the cell. Consequently, using SageMath's `Polyhedron' class, we can efficiently determine the lower-dimensional cells and their combinatorial types by querying the centers of the boundary faces of the polyhedron. This also directly yields the adjacency relations among these cells. In this way, starting from each top-dimensional cell, we determine all lower-dimensional cells lying in its boundary.

	Transitions between cells of the same dimension are described by Whitehead moves on the corresponding Voronoi graphs. As it is possible for a face of a polytope to have a symmetry that does not extend to the entire polytope, a-priori this process is not guaranteed to uniquely determine all gluing maps, but it did suffice in our case.

Due to the fact that some cells have nontrivial automorphism groups, to obtain a cellular decomposition of moduli space, we further need to determine a fundamental domain for each cell. In principle, this can be achieved by computing the barycentric subdivision. However, this introduces an excessive number of cells. In our implementation, we therefore use a more targeted strategy. We observed that the automorphism groups of the polyhedra are generated only by reflections and rotations, so we subdivided the cells according to the following rules: For a reflection, we cut the cell along the codimension-one fixed-point set. For a rotation, we consider the orbit of each edge under the group action, choose a representative edge in each orbit, and impose inequalities requiring the angular parameter of the representative edge to be greater than or equal to those of the other edges in the same orbit. This breaks the rotational symmetry.

	A special situation occurs when a boundary face has a larger automorphism group than the cell containing it. In this case, we manually construct a barycentric refinement of the cells containing that face. In SageMath, this is implemented by choosing barycenters and using them to construct hyperplane cuts of the polyhedral cells. We begin with hyperplanes whose fixed-point sets are nonempty, and then propagate the induced subdivision step by step to higher-dimensional cells.
	
When constructing fundamental domains, there were cases in which the automorphism group of a polyhedron contained a reflection of the polyhedron. This gives rise to fundamental domains of polyhedra with boundary cells lying along fixed point sets of these reflections, rather than on the boundary of the polyhedron. We call such faces ``mirror faces''. Unlike the cells on the boundary of the fundamental domain that are also contained in the boundary of the polyhedron, mirror faces are not matched with other faces of the cell complex, and appear as boundary cells of moduli space. After a fundamental domain has been chosen, any isolated internal mirror face is recorded as a special boundary face in the corresponding SageMath $\Delta$-complex object.

	\subsection{The critical points described via angle coordinates}

\label{criticalanglesec}
	From the cellular decomposition associated with the angle parameters, we can identify the four critical points described above, as illustrated in Figure \ref{figcriticalpoints}.
	\begin{figure}
		\centering
		\includegraphics[width=\textwidth]{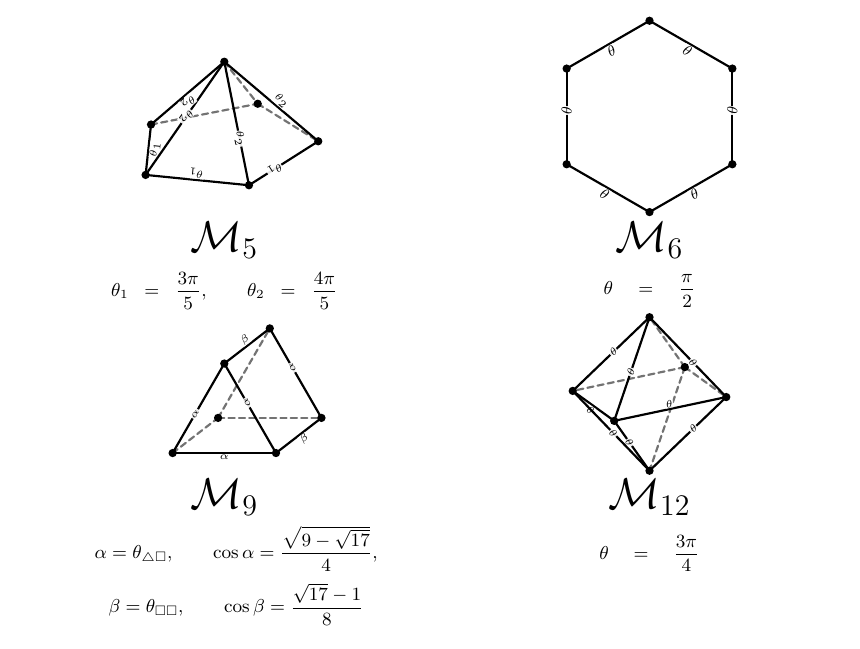}
		\caption{The Delaunay graphs of the four critical points, with edges labelled by the angle parameters at the critical point. Note that these graphs are identical to the systolic graphs of the corresponding points in $\mathcal{T}$, except for $\mathcal{M}_{5}$, where the Delaunay graph contains the systolic graph as a subgraph.}
		\label{figcriticalpoints}
	\end{figure}

\section{Appendix: Pseudocode}

    Previously in Section \ref{jiexing}, we have described how to obtain the top-dimensional cells of a cell decomposition based on angle coordinates. Below we present the pseudocode of the algorithm.
	\newpage

    \noindent\textbf{Algorithm 1}. Compute top-dimensional cells of $\mathcal{M}_2$\label{alg:topcells}
    \begin{algorithmic}[1]
    \Require None
    \Ensure $C_6$, the set of realizable top-dimensional cells

    \State $V \gets \{0,\dots,7\}$ \Comment{Vertex set}

    \Statex

    \Comment{Step 1: Initialize admissible local configurations at vertex $0$}
    \State $\mathcal{L}_0 \gets \{$
    \Statex \hspace{1.5em} (i) $0$ connected to $1,2,3$,
    \Statex \hspace{1.5em} (ii) $0$ connected to $1$ by a double edge and to $2$ by a single edge 
    \Statex $\}$

    \Statex

    \Comment{Step 2: DFS construction of trivalent connected graphs}
    \State $\mathcal{G}_{\mathrm{cand}} \gets \varnothing$ \Comment{Candidate Voronoi graphs}
    \For{each $\ell \in \mathcal{L}_0$}
    \State Initialize partial graph $G$ from $\ell$
    \State $\mathcal{S} \gets \{G\}$ \Comment{DFS stack}
    \While{$\mathcal{S} \neq \varnothing$}
        \State Pop $G$ from $\mathcal{S}$
        \If{$G$ violates trivalence or contains loops/triple edges}
        \State \textbf{continue} \Comment{Prune}
        \EndIf
        \If{$G$ is disconnected or non-planar}
        \State \textbf{continue} \Comment{Prune}
        \EndIf
        \If{$G$ is complete (all vertices degree $3$)}
        \State $\mathcal{G}_{\mathrm{cand}} \gets \mathcal{G}_{\mathrm{cand}} \cup \{G\}$ 
        \Else
        \For{each admissible edge addition $e$}
            \State Push $G \cup \{e\}$ to $\mathcal{S}$
        \EndFor
        \EndIf
    \EndWhile
    \EndFor

    \Statex

    \Comment{Step 3: Remove isomorphic duplicates}
    \State $\mathcal{G}_{\mathrm{iso}} \gets \varnothing$ \Comment{Non-isomorphic Voronoi graphs}
    \For{each $G \in \mathcal{G}_{\mathrm{cand}}$}
    \If{$G$ is not isomorphic to any graph in $\mathcal{G}_{\mathrm{iso}}$}
        \State $\mathcal{G}_{\mathrm{iso}} \gets \mathcal{G}_{\mathrm{iso}} \cup \{G\}$
    \EndIf
    \EndFor
    \Comment{Result: $|\mathcal{G}_{\mathrm{iso}}| = 17$}

    \Statex

    \Comment{Step 4: Enumerate rotation systems (combinatorial embeddings)}
    \State $\mathcal{E} \gets \varnothing$ \Comment{Admissible embeddings}
    \For{each $G \in \mathcal{G}_{\mathrm{iso}}$}
    \For{each rotation system $\rho$ on $G$}
        \State Construct faces via dart structure
        \State Compute $F(\rho)$
        \If{$|V| - |E(G)| + F(\rho) = 2$}
        \State $\mathcal{E} \gets \mathcal{E} \cup \{(G,\rho)\}$
        \EndIf
    \EndFor
    \EndFor

    \Statex

    \Comment{Step 5: Form Voronoi--Delaunay pairs up to equivalence}
    \State $\mathcal{P}_{\mathrm{VD}} \gets \varnothing$ \Comment{Voronoi--Delaunay pairs}
    \For{each $(G,\rho) \in \mathcal{E}$}
    \State $D \gets$ dual graph of $(G,\rho)$
    \If{$(G,D)$ is not isomorphic to any pair in $\mathcal{P}_{\mathrm{VD}}$}
        \State $\mathcal{P}_{\mathrm{VD}} \gets \mathcal{P}_{\mathrm{VD}} \cup \{(G,D)\}$
    \EndIf
    \EndFor
    \Comment{Result: $|\mathcal{P}_{\mathrm{VD}}| = 20$}

    \Statex

    \Comment{Step 6: Geometric realizability test}
    \State $C_6 \gets \varnothing$ \Comment{Output set}
    \For{each $(G,D) \in \mathcal{P}_{\mathrm{VD}}$}
    \State Assign angle parameters and check constraints
    \If{constraints admit a hyperbolic solution}
        \State $C_6 \gets C_6 \cup \{(G,D)\}$
    \EndIf
    \EndFor
    \Comment{Result: $|C_6| = 10$}

    \State \Return $C_6$
    \end{algorithmic}

\noindent\textbf{Step 2.} Determine all cells and gluing maps.

Starting from the top-dimensional cells, we reconstruct the cell decomposition dimension by dimension. In contrast to constructing and quotienting the faces of each cell separately, we first collect all faces of a fixed dimension and then choose canonical representatives simultaneously. This allows the identification of cells and the determination of gluing maps to be carried out in a single procedure.

Let $C_d$ denotes the set of canonical DFV tuples representing $d$-dimensional cells. The resulting decomposition is essentially described by the following data:

A set $C=\bigcup_d C_d$ of canonical DFV tuples
\[
    x=(D,F,V),
\]
one for each cell, where $D$ denotes the Delaunay planar graph, $F$ the set of oriented faces of $D$, and $V$ the Voronoi dual of $D$.

For every occurrence $c$ of a polyhedral cell, its \emph{canonical image}
\[
    \operatorname{Can}(c)=(\bar c,\varphi),
\]
where $\bar c\in C$ is the chosen canonical representative of its orientation-preserving planar isomorphism class and
\[
    \varphi:\bar c\longrightarrow c
\]
is an orientation-preserving label map.

For every $d$-dimensional cell $x$, we record
\[
    T_x=\{\bar c:\bar c\text{ occurs as a canonical }(d-1)\text{-face of }x\}
\]
and
\[
    P_x=\{(\bar c,\varphi):\operatorname{Can}(c)=(\bar c,\varphi),
    \ c\text{ is a polyhedral facet of }x\}.
\]
Thus $T_x$ records the combinatorial boundary of $x$, while $P_x$ records the actual attaching maps of its polyhedral facets.

Each top-dimensional cell is regarded as a \emph{root cell}. Faces obtained from it are always represented as faces of the same root polyhedron. In particular, once the DFV tuple associated with a face of a root polyhedron has been determined, all its descendants refer back to this representative rather than reconstructing the same polyhedral face repeatedly. This gives a consistent identification of faces inside each top-dimensional polyhedron.

The decomposition can then be reconstructed by the following recursive procedure.

\noindent\textbf{Algorithm 2}. Solve the cell decomposition of $\mathcal{M}_2$
\label{alg:mbar2}

\begin{algorithmic}[1]

\Require $C_6$

\Ensure $C=\bigcup_{d=0}^6 C_d$ and a dictionary
$\operatorname{Dict}_C:x\mapsto(T_x,P_x)$

\State Mark every $x\in C_6$ as a root cell
\State $d\gets 6$

\While{$d\geq 1$}

\State $R_{d-1}^{\mathrm{int}}\gets\varnothing$
\State $R_{d-1}^{\infty}\gets\varnothing$
\Comment{Raw interior and degenerate face occurrences}

\For{each $x=(D,F,V)\in C_d$}

    \State $\mathcal{F}_x\gets$ the $(d-1)$-dimensional facets of the polyhedron represented by $x$
    \Comment{\hyperref[mbaremark1]{(1)}}

    \For{each $f\in\mathcal{F}_x$}

        \State $E_\pi(f)\gets
        \{e:\theta_e=\pi\text{ identically on }f\}$

        \State $c_f\gets\operatorname{DFV}(f)$, obtained by deleting the edges in $E_\pi(f)$
        \Comment{\hyperref[mbaremark2]{(2)}}

        \If{$f$ contains an angle coordinate $\theta_e=0$}

            \State $R_{d-1}^{\infty}
            \gets
            R_{d-1}^{\infty}\cup\{(x,f,c_f)\}$

        \ElsIf{the edges of $V$ dual to $E_\pi(f)$ contain a cycle}

            \State $R_{d-1}^{\infty}
            \gets
            R_{d-1}^{\infty}\cup\{(x,f,c_f)\}$
            \Comment{\hyperref[mbaremark3]{(3)}}

        \Else

            \State $R_{d-1}^{\mathrm{int}}
            \gets
            R_{d-1}^{\mathrm{int}}\cup\{(x,f,c_f)\}$

        \EndIf

    \EndFor

\EndFor

\State $C_{d-1}\gets\varnothing$

\For{each $(x,f,c)\in R_{d-1}^{\mathrm{int}}$}
\Comment{Choose canonical representatives globally}

    \State $\mathrm{found}\gets\mathrm{false}$

    \For{each $\bar c\in C_{d-1}$}

        \If{$\bar c$ is orientation-preserving planar isomorphic to $c$
        with $\varphi(\bar c)=c$}

            \State $\operatorname{Can}(c)\gets(\bar c,\varphi)$
            \State $\mathrm{found}\gets\mathrm{true}$
            \State \textbf{break}

        \EndIf

    \EndFor

    \If{$\mathrm{found}=\mathrm{false}$}

        \State Mark $c$ as canonical
        \State $\operatorname{Can}(c)\gets(c,\operatorname{id}_c)$
        \State $C_{d-1}\gets C_{d-1}\cup\{c\}$

    \EndIf

\EndFor

\Comment{Apply the same canonicalization procedure to
$R_{d-1}^{\infty}$ if the boundary strata of the bordification are to be retained}

\For{each $x\in C_d$}

    \State $T_x\gets\varnothing$
    \State $P_x\gets\varnothing$

    \For{each $(x,f,c)$ occurring in
    $R_{d-1}^{\mathrm{int}}\cup R_{d-1}^{\infty}$}

        \State $(\bar c,\varphi)\gets\operatorname{Can}(c)$

        \State $T_x\gets T_x\cup\{\bar c\}$

        \State $P_x\gets P_x\cup\{(\bar c,\varphi)\}$

    \EndFor

    \State \textbf{append}
    $\{x\mapsto(T_x,P_x)\}$
    \textbf{to} $\operatorname{Dict}_{C_d}$

\EndFor

\State $d\gets d-1$

\EndWhile

\State $C\gets\bigcup_{d=0}^6 C_d$

\State $\operatorname{Dict}_C
\gets\bigcup_{d=1}^6\operatorname{Dict}_{C_d}$

\end{algorithmic}

{\footnotesize

\noindent\textbf{Remarks:}

\begin{enumerate}
\renewcommand{\labelenumi}{(\arabic{enumi})}

\item
\label{mbaremark1}
For every top-dimensional DFV tuple $x$, the corresponding polyhedron is obtained directly from the angular constraints described in Section~\ref{jiexing}. Once a top-dimensional polyhedron has been constructed, all of its lower-dimensional faces are treated as faces of this fixed root polyhedron. Their DFV data are derived from the vanishing or saturated angle coordinates. Thus lower-dimensional cells need not be reconstructed independently from the angle inequalities.

\item
\label{mbaremark2}
Let $\theta_e$ denote the angle coordinate associated with a Delaunay edge $e$. If $\theta_e=\pi$ on a face, the corresponding two circumcircles coincide, and the edge $e$ disappears from the Delaunay decomposition. Hence the DFV tuple associated with a polyhedral face is obtained by deleting all such edges and merging the corresponding Delaunay faces, or equivalently by contracting their dual Voronoi edges.

The resulting DFV tuple depends only on the set of surviving edge labels. We therefore store, for each root cell, a dictionary from the surviving edge set to the corresponding DFV face. Descendants of the same root cell refer to this dictionary, which prevents the same polyhedral face from being reconstructed in different ways.

\item
\label{mbaremark3}
There are two types of degenerations detected at this stage. If some angle coordinate satisfies $\theta_e=0$, the corresponding Delaunay edge has zero length, producing a non-separating degeneration. On the other hand, deleting the edges with $\theta_e=\pi$ may contract a cycle in the Voronoi graph; this corresponds to a separating degeneration. Faces of either type belong to the boundary of the bordification rather than to $\mathcal{M}_2$ itself.

\item
The same canonical DFV tuple may occur several times among the facets of a single polyhedral cell. These occurrences need not determine the same attaching map. For this reason $T_x$ records only the distinct combinatorial faces, whereas $P_x$ retains every pair $(\bar c,\varphi)$ and hence contains the information necessary to reconstruct the polyhedral gluing.

\end{enumerate}

\section{Use of AI tools}
In Section  \ref{spinecomputation}, Kimi k2.5 was used to derive some code for solving equations, numerical calculation of 1st‑ to 5th‑order function derivatives, linear and quadratic programming, and gradient descent of convex functions. ChaptGPT 5 was used to tidy up the code from Section \ref{jiexingcomp} (this code was already functional without AI in \cite{JChen}) and for polishing the English in section \ref{jiexingcomp}.

}
	
	\bibliography{spinebib2}
	\bibliographystyle{plain}
	
\end{document}